\documentclass[11pt,reqno]{amsart}
\usepackage{amsmath,amsthm,amssymb,mathrsfs,stmaryrd,color}
\usepackage[all]{xy}
\usepackage{url}
\usepackage{geometry}
\usepackage[utf8]{inputenc}
\usepackage[T1]{fontenc}

\usepackage{relsize}
\usepackage[bbgreekl]{mathbbol}
\usepackage{amsfonts}

\DeclareSymbolFontAlphabet{\mathbb}{AMSb}
\DeclareSymbolFontAlphabet{\mathbbl}{bbold}

\usepackage{enumitem}
\usepackage[colorlinks=true,hyperindex, linkcolor=magenta, pagebackref=false, citecolor=cyan,pdfpagelabels]{hyperref}
\usepackage[capitalize]{cleveref}
\usepackage{mathtools}
\usepackage{tikz-cd}
\newsavebox{\pullback}
\sbox\pullback{%
\begin{tikzpicture}%
\draw (0,0) -- (1ex,0ex);%
\draw (1ex,0ex) -- (1ex,1ex);%
\end{tikzpicture}}
\setlist[enumerate]{itemsep=2pt,parsep=2pt,before={\parskip=2pt}}
\setlist[enumerate]{itemsep=2pt,parsep=2pt,before={\parskip=2pt}}

\newcommand{\cosimp}[3]{\xymatrix@1{#1 \ar@<.4ex>[r] \ar@<-.4ex>[r] & {\ }#2 \ar@<0.8ex>[r] \ar[r] \ar@<-.8ex>[r] & {\ } #3 \ar@<1.2ex>[r] \ar@<.4ex>[r] \ar@<-.4ex>[r] \ar@<-1.2ex>[r] & \cdots }}

\newcommand{\colim}{\mathop{\mathrm{colim}}}
\newcommand{\Pic}{\mathop{\mathrm{Pic}}}
\newcommand{\adjunction}[4]{\xymatrix@1{#1{\ } \ar@<0.3ex>[r]^{ {\scriptstyle #2}} & {\ } #3 \ar@<0.3ex>[l]^{ {\scriptstyle #4}}}}

\newcommand{\ul}[1]{\underline{#1}}

\newtheorem{theorem}{Theorem}[section]
\newtheorem{conjecture}[theorem]{Conjecture}
\newtheorem*{theorem*}{Theorem}
\newtheorem*{definition*}{Definition}
\newtheorem{proposition}[theorem]{Proposition}
\newtheorem{lemma}[theorem]{Lemma}
\newtheorem{corollary}[theorem]{Corollary}

\theoremstyle{definition}
\newtheorem{definition}[theorem]{Definition}
\newtheorem{question}[theorem]{Question}
\newtheorem{remark}[theorem]{Remark}

\newtheorem{example}[theorem]{Example}
\newtheorem{notation}[theorem]{Notation}
\newtheorem{convention}[theorem]{Convention}
\newtheorem{construction}[theorem]{Construction}

\crefname{assumption}{assumption}{assumptions}
\crefname{construction}{construction}{constructions}

\newcommand{\Spd}{\mathrm{Spd}}

\newcommand{\et}{\mathrm{\acute{e}t}}
\newcommand{\proet}{\mathrm{pro\acute{e}t}}
\newcommand{\profet}{\mathrm{prof\acute{e}t}}

\newcommand{\perfd}{\mathrm{perfd}}

\newcommand{\Id}{\mathrm{Id}}

\newcommand{\rig}{\mathrm{rig}}

\DeclareMathOperator{\Lie}{Lie}
\DeclareMathOperator{\Hom}{Hom}

\DeclareMathOperator{\Spa}{Spa}

\title{Characterizing perfectoid covers of abelian varieties}
\author{Rebecca Bellovin, Hanlin Cai, and Sean Howe \\with an appendix by Tongmu He}

\begin{document}

\begin{abstract}
We give a simple characterization of all perfectoid profinite \'{e}tale covers of abelian varieties in terms of the Hodge-Tate filtration on the $p$-adic Tate module. We also compute the geometric Sen morphism for all profinite $p$-adic Lie torsors over an abelian variety, and combine this with our characterization to prove a conjecture of Rodr\'{i}guez Camargo on perfectoidness of $p$-adic Lie torsors in this case. We obtain complementary results for covers of semi-abeloid varieties, $p$-divisible rigid analytic groups, and varieties with globally generated 1-forms. Our proof of perfectoidness for covers of abelian varieties is based on results of Scholze on the canonical subgroup and holds for an arbitrary abelian variety over an algebraically closed non-archimedean extension of $\mathbb{Q}_p$. In an appendix authored by Tongmu He, an alternate proof is presented in the case of abelian varieties that can be defined over a discretely valued subfield by combining our computation of the geometric Sen morphism with previous pointwise perfectoidness and purity of perfectoidness results of He. 
\end{abstract}

\maketitle

\tableofcontents

\section{Introduction}

Fix a prime $p$. Perfectoid spaces are a class of spaces in $p$-adic geometry introduced in~\cite{ScholzePerfectoidspaces} to study profinite covers of rigid analytic spaces with a lot of $p$-power ramification.  They generalize the deeply ramified extensions of discretely valued fields studied by Fontaine and Wintenberger, and have come to play a similarly important role in relative $p$-adic Hodge theory, the internal study of $p$-adic geometry, and applications to problems in number theory. 

It is thus natural to ask:

\begin{question}\label{question.which-perfectoid} For $K/\mathbb{Q}_p$ a complete non-archimedean field extension and $X/K$ a connected rigid analytic space, which connected profinite \'{e}tale covers of $X$ are perfectoid?
\end{question}

Historically, the main technique for proving perfectoidness for such a cover has been to construct an integral model for a related tower of finite \'{e}tale covers where the transition morphisms factor through lifts of Frobenius. This technique applies, in particular, to perfectoid tori over perfectoid fields, which provide a source of perfectoid covers of toric varieties as in \cite{ScholzePerfectoidspaces} that can also be used to construct perfectoid covers locally on smooth rigid analytic varieties.  These covers, in turn, provide coordinates suitable for explicit computations in $p$-adic Hodge theory as in, e.g., \cite{scholze2013adic}. For perfectoid tori, the associated tower of finite \'{e}tale covers is given by multiplication by $p$ on a group variety, and the same technique was used also to prove perfectoidness for the analogous covers of $p$-divisible rigid analytic groups in \cite{scholze2012moduli} and abelian varieties in \cite{Blakestad}. 

At a deeper level, this technique was applied to established perfectoidness of infinite level Hodge type Shimura varieties in \cite{ScholzeOntorsioninthecohomologyoflocallysymmetricvarieties}.  Here the Frobenius lifts arise naturally in a part of the tower from a study of overconvergence of the canonical subgroup, and it is necessary to bootstrap the argument to see that the infinite level Shimura varieties with full level structure are perfectoid.

Recently, Rodriguez Camargo has proposed a criterion \cite[Conjecture 3.3.5]{JuangeometricSen} that relates perfectoidness of certain covers to geometric Sen theory, and He \cite{he2024perfd} has proven a pointwise version of this conjecture when the base $X$ is the analytification of an algebraic variety.

{\bf In this paper we restrict to the case when $K=C$ is algebraically closed, and we give a complete answer to \cref{question.which-perfectoid} when $X$ is an abelian variety.} We also obtain complementary results for covers of semi-abeloid varieties and $p$-divisible rigid analytic groups (in the sense of \cite{Fargues-RigAnPDiv}), and relate our results and methods to \cite[Conjecture 3.3.5]{JuangeometricSen} and \cite{he2024perfd}; in particular, we compute the geometric Sen morphism for compact $p$-adic Lie torsors over abelian varieties and prove \cite[Conjecture 3.3.5]{JuangeometricSen} in this case. The proof of our main perfectoidness result for abelian varieties depends on results on canonical subgroups established in the proof of perfectoidness of the infinite level Siegel moduli space in \cite{ScholzeOntorsioninthecohomologyoflocallysymmetricvarieties}; in \cref{appendix.he}, Tongmu He gives an alternate proof of a special case using the purity of perfectoidness established in \cite{he2024purity}. 

\subsection{A precise question.}\label{ss.precise-question}
Before describing our results in detail, we make \cref{question.which-perfectoid} more precise in the case where $K=C$ is algebraically closed. For a general connected rigid analytic space $X/C$ equipped with a point $x \in X(C)$, there is a profinite \'{e}tale fundamental group $\pi_{1}(X,x)$, and an associated diamond $X^\diamond$ over $\Spd C$ in the sense of \cite{ScholzeBerkeleyLectures}, which we interpret as the functor of points of $X$ restricted to perfectoid spaces over $C$. For any open subgroup $H \leq \pi_{1}(X,x)$, there is a finite \'{e}tale cover $X_H/X$ in the category of rigid analytic varieties over $C$. For a closed subgroup $H \leq \pi_{1}(X,x)$, we define $X_H^\diamond:=\lim_{H \leq H' \textrm{open}} X_{H'}^\diamond$. Then any connected profinite \'{e}tale cover of $X^\diamond$ is of the form $X_H^\diamond$ for a unique closed subgroup $H \leq \pi_1(X,x)$.  For $e \in H$ the identity element, $X_{\{e\}}^\diamond$ is the universal profinite \'{e}tale cover cover of $X^\diamond$ and we have $X_H^\diamond=X_{\{e\}}^\diamond/H$ for any closed subgroup $H \leq \pi_1(X,x)$. Note that, when $H$ is not open, the limit cannot be taken in the category of adic spaces, so there is no natural adic space $X_H$. In this case, we may still write $X_H$ for the tower $(X_{H'})_{H \leq H' \textrm{ open}}$ viewed as an object of the pro\'{e}tale site $X_\proet$ defined in \cite{scholze2013adic}. It makes sense to associate a diamond to such an object, and the notation is compatible with what we have called $X_H^\diamond$ above. 

We say that $X_H^\diamond$ is perfectoid if it is representable, i.e. if it is isomorphic to the functor of points of a perfectoid space over $C$. We say that $X_H$ is perfectoid if it is perfectoid as an object of $X_\proet$ in the sense of  \cite[Definition 4.3]{scholze2013adic}. This is a strictly stronger condition than asking that $X_H^\diamond$ be perfectoid (see \cref{remark.distinguishing-perfectoid-conditions} for an example distinguishing the two): essentially, one is asking not only that $X_H^\diamond$ be represented by a perfectoid space, but also that there is a covering of $X_H^\diamond$ by affinoid perfectoids such that functions on those affinoid perfectoids can be approximated arbitrarily well by functions on open affinoids of the finite level spaces $X_{H'}$. In particular, if $X_H$ is perfectoid, then $X_H^\diamond$ is represented by a perfectoid space $P \sim \lim_{H \leq H' \textrm{ open}} X_{H'}$, for the $\sim$-limit of adic spaces as defined in \cite[Definition 2.4.1]{scholze2012moduli}. The precise version of \cref{question.which-perfectoid} we treat is then:
\begin{question}For $C/\mathbb{Q}_p$ an algebraically closed non-archimedean extension, $X/C$ a connected rigid analytic space, and $x \in X(C)$, for which closed $H\leq \pi_1(X,x)$ is $X_H$ and/or $X_H^\diamond$ perfectoid?
\end{question}

\emph{We fix an algebraically closed non-archimedean extension $C/\mathbb{Q}_p$ for the remainder of the paper.} 

\subsection{Perfectoid covers of abelian varieties}
For $A/C$ an abelian variety, the universal profinite \'{e}tale cover is $A_{\{e\}}=\widetilde{A} \coloneqq (A)_{n \in \mathbb{N}}$, where $\mathbb{N}$ is ordered by divisibility and the transition map from $n$ to $m$ with $m|n$ is given by multiplication by $n/m$, and $\pi_1(A,e)=T_{\widehat{\mathbb{Z}}}A\coloneqq\lim_n A[n](C)$ acting on $\widetilde{A}$ by multiplication. By the Chinese remainder theorem, any closed subgroup $H\leq T_{\widehat{\mathbb{Z}}}A$ factors uniquely as a product $\prod_\ell H_\ell$, for closed subgroups $H_\ell \leq T_\ell A = \lim A[\ell^k](C)$, and the key role is played by the factor $H_p$. In particular, the results of \cite{Blakestad} together with almost purity imply that, if $H_p=\{e\}$, then $A_H$ is perfectoid --- for example, $\widetilde{A}$ is perfectoid. 

Our main result, \cref{theorem.cover-classification}, is a simple characterization of all perfectoid profinite \'{e}tale covers of an abelian variety $A/C$. In fact, we will state a more general result allowing semi-abelian varieties (i.e. extensions of abelian varieties by tori) and $p$-divisible rigid analytic groups (in the sense of Fargues \cite{Fargues-RigAnPDiv}). For $E$ a semi-abelian variety or $p$-divisible rigid analytic group we will also write $\widetilde{E} \coloneqq (E)_n$, a profinite \'{e}tale $T_{\widehat{\mathbb{Z}}}E\coloneqq\lim_n E[n](C)$ cover of $E$. We emphasize that when $E$ is not an abelian variety this may not be the universal profinite \'{e}tale cover, and we will only consider covers which are quotients of $\widetilde{E}$ (and thus correspond to closed subgroups of $T_{\widehat{\mathbb{Z}}}E$). 

\begin{remark} For $G$ a $p$-divisible rigid analytic group, multiplication by $n$ on $G$ is invertible for any $n$ coprime to $p$, so $T_{\widehat{\mathbb{Z}}} G = T_p G$ and $\widetilde{G}=\lim_{[p]_G} G^\diamond$ via the natural maps. There is thus no conflict with the notation of \cite{scholze2012moduli}. 
\end{remark}

In order to state the determining criterion, we note that, for $G/C$ a semi-abelian variety or $p$-divisible rigid analytic group, there is a canonical map $f: \Lie G  \rightarrow T_p G \otimes C(-1)$. It is an injection in the semi-abelian case and in the case of $p$-divisible rigid analytic groups with no vector group component. By \cite[Theorem 5.2.1]{scholze2012moduli}, the latter are precisely the $p$-divisible rigid analytic groups arising as the generic fibers of $p$-divisible groups over $\mathcal{O}_C$, and in this case $f$ is the dual Hodge-Tate map of \cite{tate1967p-divisible}. We refer to the map $f$ in general as the dual Hodge-Tate map; by the above discussion, when $G$ is a semi-abelian variety or a $p$-divisible rigid analytic group with no vector component, $f$ is simply the inclusion of the subspace in the Hodge-Tate filtration on $T_p G \otimes C(-1)$. 

\begin{theorem}\label{theorem.cover-classification}Let $C/\mathbb{Q}_p$ be an algebraically closed non-archimedean extension and let $G/C$ be either a semi-abelian variety or a $p$-divisible rigid analytic group. Let $H=\prod H_\ell \leq T_{\widehat{\mathbb{Z}}}G$ be a closed subgroup and let $f: \Lie G \rightarrow T_p G \otimes C(-1)$ be the dual Hodge-Tate map. If  $f^{-1}(H_p \otimes C(-1))=\{0\}$, then $G_H$ (and thus also $G_H^\diamond$) is perfectoid. Conversely, if $f^{-1}(H_p\otimes C(-1))\neq \{0\}$, then no open subdiamond of $G_H^\diamond$ (and thus no open in $G_H$) is represented by a perfectoid space.  
\end{theorem}

   The case of $p$-divisible rigid analytic groups is \cref{prop.cover-classification-p-divisible-group}. For semi-abelian varieties, the perfectoidness statement is \cref{prop.perfectoid-covers-abelian}, and the non-perfectoidness statements are contained in \cref{prop.cover-classification-semi-abeloid} (which treats, more generally, semi-abeloid varieties; see \cref{remark.semi-abeloid}).

In order to prove \cref{theorem.cover-classification} in full, we first treat the case of $p$-divisible rigid analytic groups.  Their structure theory (due to Fargues \cite{Fargues-RigAnPDiv}) allows us to give a very explicit proof for $p$-divisible rigid analytic groups $G$ by reduction to the explicit case of the $p$-divisible rigid analytic group of a torus. For $A$ a semi-abelian variety and $G \leq A$ its $p$-divisible rigid analytic group (consisting of elements on which multiplication by $p$ is topologically nilpotent), we then obtain the non-perfectoidness result for $A$ from the non-perfectoidness result for $G$ by considering the multiplication action of $\widetilde{G}$ on $\widetilde{A}$. 

However, the perfectoidness result for $G$ is not sufficient to give the perfectoidness result for $A$: indeed, it only implies the (a priori) weaker result that there is an open neighborhood $U$ of the diagonal in $A \times A$ such that projection to the second factor $(A_{H} \times A_H)|_{U} \rightarrow A_H$ is relatively representable in perfectoids (in particular, for any $t: T \rightarrow A_H^\diamond$ for $T$ perfectoid, the restriction of $A_H^\diamond$ to the neighborhood $U_t^\diamond$ of $t$ in $A_T^\diamond$ is perfectoid).

To prove the full perfectoidness statement in the semi-abelian case we need a different strategy. Similarly to \cite{Blakestad}, we use Raynaud uniformization to reduce to the case of an abelian variety of good reduction.  In the good reduction case, we first use some basic facts about abelian varieties and perfectoid spaces to reduce to the case that $A$ admits a principal polarization and $H_p$ is a maximal isotropic subspace. In this case, we appeal to some results of \cite{ScholzeOntorsioninthecohomologyoflocallysymmetricvarieties} relating canonical subgroups and the Hodge-Tate period map to show the injectivity implies that the tower determined by $H_p$ is eventually given by a sequence of Frobenius lifts --- indeed, up to an innocent isogeny, we will see that the associated point $(A, H_p)$ of the Siegel modular variety at level $\Gamma_0(p^\infty)$ lies in the open anti-canonical locus as considered in \cite{ScholzeOntorsioninthecohomologyoflocallysymmetricvarieties}. Such a tower of Frobenius lifts is perfectoid, so we conclude. Note that the case of $H_p=\{e\}$ that is handled by the results of \cite{Blakestad} does not need this finer analysis after reducing to the good reduction case since multiplication by $p$ always factors through the Frobenius.

\begin{remark}For $A/C$ an abelian variety, by dimension considerations the condition for perfectoidness of \cref{theorem.cover-classification} never holds when $\mathrm{rk} H_p > \dim A$, and thus in such a case $A_H^\diamond$ is not perfectoid and thus $A_H$ is also not perfectoid. This can also be seen  by computing the pro-\'etale cohomology of the completed structure sheaf using the natural spectral sequence for the covering $A_{H} \rightarrow A$ --- if $A_{H}$ were perfectoid and $T_p A/ H_p$ had rank less than $\dim A$, then an analysis of the $E_2$ page (using vanishing of the analytic cohomology of $A_H$ in degree above $\dim A$ and vanishing of continuous group cohomology in degree above the rank of $T_p A/H_p$) would imply that $H^{2d}(A_\proet, \widehat{\mathcal{O}}_A)=0$ where $d=\dim A$, when in fact it is equal to $C$ by the primitive comparison theorem.
\end{remark}

\begin{remark}\label{remark.semi-abeloid} We establish the non-perfectoidness part of \cref{theorem.cover-classification} more generally for semi-abeloid varieties in \cref{prop.cover-classification-semi-abeloid}. We also obtain relative perfectoidness of a neighborhood of the diagonal for semi-abeloid varieties, but the proof of perfectoidness for semi-abelian varieties sketched above depends on a relation between the Hodge-Tate filtration and canonical subgroups in the case of principally polarized abelian varieties of good reduction that is established in \cite{ScholzeOntorsioninthecohomologyoflocallysymmetricvarieties}. The proof of this relation uses the existence and nice properties of the moduli space of principally polarized abelian varieties, so does not obviously extend to abeloid varieties. Nonetheless, we expect that this relation should hold for all abeloid varieties of good reduction (see \cref{remark.abeloid-canonical-subgroup-expected-statement}).  
\end{remark}

\subsection{The geometric Sen morphism}

We return for a moment to a general smooth rigid analytic space $X/C$. We have already seen above the special role played by the $p$-adic part of the profinite \'{e}tale fundamental group when $X$ is an abeloid variety. In general, it is natural to single out this $p$-adic part by considering compact $p$-adic Lie groups $K$ and pro-\'{e}tale $K$-torsors $X_\infty / X$, corresponding to continuous homomorphisms $\rho: \pi_1(X,x) \rightarrow K$. The geometric Sen theory of Pan \cite{pan2022locally} and Rodriguez Camargo \cite{JuangeometricSen} produces from $X_\infty$ a canonical Higgs field, which we view as a morphism on $X_\proet$
\[ \kappa_{X_\infty}: T_{X/C} \rightarrow (\Lie K)_\rho \otimes_{\mathbb{Q}_p} {\widehat{\mathcal{O}}_X}(-1), \]
where $T_X$ is the tangent bundle of $X$, viewed as an $\widehat{\mathcal{O}}_X$-module by the fully faithful embedding of \'{e}tale vector bundles into locally free $\widehat{\mathcal{O}}_X$-modules, and $(\Lie K)_\rho$ is appearing as a $\underline{\mathbb{Q}_p}$-local system in its twisted form by the adjoint action of $\pi_1(X,x)$ through $\rho$. If $\rho$ factors through an abelian subgroup $H \leq K$ then $\kappa_{X_\infty}$ factors through $(\Lie H)_\rho= \Lie H$ (i.e. the twist is trivial on $\Lie H$ because the adjoint action is trivial on $\Lie H$); this is the situation in all of our main results appearing below. 

In \cite[Conjecture 3.3.5]{JuangeometricSen}, it is conjectured that $X_\infty$, viewed as an object of $X_\proet$, is perfectoid if and only if $\kappa_{X_\infty}^*$ is surjective, or equivalently if and only if $\kappa_{X_\infty}$ is injective after restriction to the fiber at any geometric point.  For a semi-abelian variety or $p$-divisible rigid analytic group $G/C$, we compute the geometric Sen morphism for all such $\rho$ factoring through the $T_{\widehat{\mathbb{Z}}}G$: it is the constant composition of the dual Hodge-Tate map with $d\rho_p$, where $\rho_p$ is the restriction of $\rho$ to $T_p G$.  Combining this computation with \cref{theorem.cover-classification} proves \cite[Conjecture 3.3.5]{JuangeometricSen} in these cases. In particular, we note that for abelian varieties all $K$-torsors are of the above form, so that this is a complete resolution of \cite[Conjecture 3.3.5]{JuangeometricSen} for abelian varieties. 

\begin{theorem}\label{theorem.geometric-sen}
Let $C/\mathbb{Q}_p$ be an algebraically closed non-archimedean extension and let $G/C$ be a semi-abelian variety or a $p$-divisible rigid analytic group. Let $K$ be a compact $p$-adic lie group. Suppose $\rho: T_{\widehat{\mathbb{Z}}} G \rightarrow K$ is a continuous homomorphism, and let $G_\infty$ be the associated $K$-torsor over $G$. Then, the geometric Sen morphism $\kappa_{G_\infty}$ is the constant composition of $d\rho_p\otimes C$  with 
$f: \Lie G \rightarrow T_p G \otimes C(-1)$, where $f: \Lie G \rightarrow T_p G\otimes C(-1)$ is the dual Hodge-Tate map, i.e.
\[ \kappa_{G_\infty} = \left( (d\rho_p \otimes C) \circ f \right) \otimes_C \widehat{\mathcal{O}}_G.\] 
If $(d\rho_p \otimes C) \circ f$ is injective, then $G_\infty$ (and thus also $G_\infty^\diamond$) is perfectoid. Otherwise, no open subdiamond of $G_\infty^\diamond$ (and thus no also no open in $G_\infty$) is perfectoid. 
\end{theorem}
\begin{proof}
The case of $p$-divisible rigid analytic groups is \cref{prop.geometric-sen-p-divisible}. For semi-abelian varieties, the computation of the geometric Sen morphism and non-perfectoidness statements are contained in \cref{prop.geometric-sen-semi-abeloid} (which treats, more generally, semi-abeloids), and the perfectoidness statement is \cref{cor.geometric-sen-perfectoid-semi-abelian}.
\end{proof}

After the computation of the geometric Sen morphism, the result follows easily from \cref{theorem.cover-classification}. To compute the geometric Sen morphism, we first treat the case of a $p$-divisible rigid analytic group, where the structure theory again reduces us to the simple case of the $p$-divisible rigid analytic group of a torus. For a semi-abelian variety $A$ with $p$-divisible rigid analytic group $G$, we obtain the computation of the geometric Sen morphism by considering the action of $G$ on $A$. 

\begin{remark}In \cite{HoweTangentPAdicMan}, \cite[Conjecture 3.3.5]{JuangeometricSen} is refined by considering natural Banach-Colmez Tangent Bundles attached to a $p$-adic Lie torsor. It is conjectured there that a torsion part of the Banach Colmez Tangent Space at a geometric point of $X_\infty$, which measures the defect from perfectoidness of the Tangent Space, is equal to a module of continuous derivations on functions at that point. Our method of proof also gives a partial result towards this refined conjecture; see also \cref{remark.distinguishing-perfectoid-conditions} for a further illustration of this point.  
\end{remark}

\begin{remark}
Just as in \cref{remark.semi-abeloid}, the computation of the geometric Sen morphism and non-perfectoidness results in \cref{theorem.geometric-sen} hold also for semi-abeloid varieties, but in the general semi-abeloid case under the injectivity assumption we obtain a weaker statement than perfectoidness. These statements are given in \cref{prop.geometric-sen-semi-abeloid}.
\end{remark}

\begin{remark}\label{remark.he-dv}
When $A$ is an abelian variety defined over a discretely valued subfield of $C$, our computation of the geometric Sen morphism combined with the results of \cite[Theorem 1.17]{he2024perfd} gives a weaker pointwise perfectoidness result in \cref{theorem.cover-classification} and \cref{theorem.geometric-sen} under the same injectivity hypothesis. Moreover, after making the same reduction to the good reduction case, in this case one can replace our argument for perfectoidness with an application of the purity of perfectoidness of \cite{he2024purity}. This is explained in detail in \cref{appendix.he}, authored by Tongmu He.  
\end{remark}

Using the Albanese morphism, we also obtain some partial results for some other varieties.

\begin{corollary}[See \cref{corollary.variety with gg diff}]
    Let $C/\mathbb{Q}_p$ be an algebraically closed non-archimedean extension and $X/C$ be a smooth proper connected variety with globally generated 1-forms. Fix a point $x \in X(C)$ and denote the corresponding morphism as $\phi:\pi_1(X,x)^\mathrm{ab}\to T_{\widehat{\mathbb{Z}}}\mathrm{Alb}(X)$. Then, for $H$ a closed subgroup of $\pi_{1}(X,x)^\mathrm{ab}$, $X_H^\diamond$ is perfectoid if $\phi(H)_p \otimes C(-1) \cap \Lie \mathrm{Alb}(X)=0$.
\end{corollary}

In some cases, such as curves of positive genus, one can obtain a stronger result --- see \cref{remark.unramified-alb-better-version} for more details.

\subsection{Outline}
In \cref{s.preliminaries} we recall some definitions and constructions related to the pro-\'{e}tale site, diamonds, perfectoid objects, and geometric Sen theory. In \cref{s.p-divisible-rigid-groups} we recall the theory of $p$-divisible rigid analytic groups and prove our main results in that case. In \cref{s.semi-abeloid-varieties} we treat the semi-abeloid and semi-abelian cases of our results --- all of our results in the general semi-abeloid case are deduced from the corresponding results for $p$-divisible rigid analytic groups, but to obtain the global perfectoidness results in the semi-abelian case requires an additional argument given in \cref{ss.perfectoid-covers}. In \cref{s.varieties-with-globally-generated-1-forms} we discuss consequences for varieties with globally generated 1-forms. In \cref{appedix.abeloid-fundamental-group} we show the profinite \'{e}tale fundamental group $\pi_1(A,e)$ of an abeloid variety $A/C$ is a $T_{\widehat{\mathbb{Z}}}A$; this is known to experts, but we include a proof for completeness. Finally, in \cref{appendix.he}, Tongmu He gives an alternate proof of perfectoidness for covers of abelian varieties as described in \cref{remark.he-dv}. 

\subsection{Acknowledgements}
During part of the preparation of this work, Rebecca Bellovin and Sean Howe were members of the Institute for Advanced Study; Rebecca Bellovin was supported by the Minerva Research Foundation, and Sean Howe was supported by the NSF through grant DMS-2201112, and by the Friends of the Institute for Advanced Study Membership. They were additionally supported as visitors at the 2023 Hausdorff Trimester on The Arithmetic of the Langlands Program by the Deutsche Forschungsgemeinschaft (DFG, German Research Foundation) under Germany’s Excellence Strategy– EXC-2047/1– 390685813.

We thank Tongmu He, Ben Heuer, Juan Esteban Rodr\'{i}guez Camargo, and Peter Wear for helpful conversations and correspondence related to this work. 

\section{Preliminaries}\label{s.preliminaries}
\begin{convention}
     Recall from the end of \cref{ss.precise-question} that we have fixed, for the entire article, a complete algebraically closed extension $C/\mathbb{Q}_p$. As in \cite{scholze2013adic}, by a rigid analytic space over $C$ we mean a locally topologically of finite type adic space over ${\Spa C}\coloneqq\Spa(C,\mathcal{O}_C)$. We will often slightly abuse notation by identifying an algebraic variety with its analytification.

     An abeloid variety over $C$ is a smooth proper connected rigid analytic group over $C$. An abeloid variety is the analytification of an abelian variety over $C$ if and only if it admits a polarization; in this case, we also refer to the abeloid variety as an abelian variety. A semi-abeloid variety (resp. semi-abelian variety) is an extension of an abeloid variety (resp. abelian variety) by a torus $T \cong \mathbb{G}_m^{\mathrm{ad},k}$. 
\end{convention}

\subsection{The pro-\'{e}tale and profinite \'etale site}\label{ss.profet-covers}

For $X/C$ a rigid analytic space, we write $X_\proet$ for the pro-\'{e}tale site of $X$ as defined in \cite[Definition 3.9]{scholze2013adic}. Recall that it is the full subcategory of the pro-objects $\mathrm{Pro}(X_\et)$ whose objects are those $U=(U_i)_{i\in I}$ such that $U_i\to U_j$ is finite \'etale and surjective for sufficiently large $i> j$. Moreover, we write $X_\profet$ for the pro-finite \'{e}tale site of $X$ as defined in \cite[Definition 3.3]{scholze2013adic}. It is the full subcategory of $X_\proet$ whose objects are pro-systems consisting of objects finite \'{e}tale over $X$. For $x \in X(C)$, we write $\pi_1(X,x)$ for the profinite \'{e}tale fundamental group of $X$. By \cite[Proposition 3.5]{scholze2013adic}, there is a canonical equivalence between the category of profinite sets with continuous $\pi_1(X,x)$-action and $X_\profet$.

Recall from \cite[Definition 4.3]{scholze2013adic} that an object $U$ of $X_\proet$ is perfectoid if it admits an open cover by objects $V$ such that each $V$ is affinoid perfectoid in the sense that it admits a presentation $V=(V_i)_{i \in I}$ for where $V_i=\Spa(R_i,R^+_i)$ and, for $R^+:=(\colim R_i^+)^\wedge$ (where $\bullet^\wedge$ is the $p$-adic completion), $(R,R^+)$ is affinoid perfectoid over $(C,\mathcal{O}_C)$. The following lemmas will be used repeatly later in the paper.

\begin{lemma}[{\cite[Lemma 4.6]{scholze2013adic}}]\label{lemma.pro-etale-cover-perf-is-perf}
    Let $X/C$ be a rigid analytic space. If $U\in X_\proet$ is perfectoid and $V\to U$ is pro-\'etale, then $V$ is perfectoid.
\end{lemma}

\begin{lemma}
\label{lem.product of perfedctoids in proet}
    Let $X\to Y \xleftarrow{} Z$ be a diagram of rigid analytic spaces over $C$ and $U=(U_i)_{i\in I}\in X_\proet$, $V=(V_i)_{i\in I}\in Y_\proet$, $W=(W_i)_{i\in I}\in Z_\proet$. Let $U\to V \xleftarrow{} W$ be a diagram of pro-objects of rigid analytic spaces. Then $U\times_V W \in (X\times_Y Z)_\proet$. Moreover, if $U$, $V$ and $W$ are perfectoid, then $U\times_V W$ is perfectoid.
\end{lemma}
\begin{proof}
    The statement that $U\times_V W \in (X\times_Y Z)_\proet$ follows formally from the fact that pro-\'etale morphisms are stable under base change and composition. For the second statement, one immediately reduces to the case where $U,V,W$ are affinoid perfectoid. Let $\widehat{U}$ be the associated affinoid perfectoid space of $U$, respectively for $\widehat{V} $ and $\widehat{W}$. Since the category of perfectoid spaces admits fibre products, it suffices to show that 
    $$ \widehat{U}\times_{\widehat{V}}\widehat{W} \sim \lim U_i\times_{V_i} W_i$$
    in the sense of \cite[Definition 2.4.1]{scholze2012moduli}. Then it follows from \cite[Proposition 2.4.2]{scholze2012moduli} since it is obvious on the level of complete affinoid $C$-algebras.
    
    \end{proof}

\begin{lemma}
\label{lem.zariski closed of perfectoid in proet}
    Let $i:Z\to X$ be a Zariski-closed immersion of smooth rigid analytic spaces over $C$. If $U=(U_i)_{i\in I}\in X_\proet$ is perfectoid then $U\times_X Z \in Z_\proet$ is perfectoid.
\end{lemma}
\begin{proof}
    We can work locally with affinoids. Let $\widehat{U}$ be the associated perfectoid space of $U$. Assume that $X=\Spa(R_0,R_0^+)$, $Z=\Spa(R_0/I,(R_0/I)^+)$, $U_i=\Spa(R_i,R_i^+)$ and $\widehat{U}=\Spa(R_\infty,R_\infty^+)$. We need to show that $(\colim (R_i/IR_i)^+)^\wedge[1/p]$ is perfectoid. First note that $(R_\infty^+/IR_\infty\cap R_\infty^+)_\perfd[1/p] $ is static by \cite[Theorem 10.11]{BhattScholzepPrismaticCohomology}. Moreover, $$R_\infty/IR_\infty \to (R_\infty^+/IR_\infty\cap R_\infty^+)_\perfd[1/p]$$ 
    is surjective by \cite[Theorem 7.4]{BhattScholzepPrismaticCohomology}. Let $S$ be $(R_\infty^+/IR_\infty\cap R_\infty^+)_\perfd[1/p]$ and let $S^+$ be the minimal integrally closed open subring that contains the image of $R_\infty^+$. Since $U_i\times_X Z$ is smooth by assumption, it is $v$-complete by \cite[Theorem 11.18]{hansen2020sheafiness}. Since the perfectoid space $\Spa(S,S^+)$ is a $v$-cover of $U_i\times_X Z$ for each $i$, the canonical map 
    $$\colim R_i/IR_i \to S$$
    is injective and the subspace topology is exactly the $p$-adic topology induced by the lattice $\colim (R_i/IR_i)^+$. However, since the image of $\colim R_i/IR_i$ is dense in $R_\infty/IR_\infty$ and the latter surjects onto $S$, the $p$-adic completion of $\colim R_i/IR_i$ is $S$ as desired.
\end{proof}

\subsection{Diamonds and covers}

We will use the notion of diamonds as defined in \cite{ScholzeBerkeleyLectures, scholze2017etale}. We will work only with perfectoid spaces and diamonds over $\Spd C$, which we view as (a subcategory of) sheaves for the $v$-topology on the category of perfectoid spaces over ${\Spa C}$. In particular, given an adic space $X/{\Spa C}$, we write $X^\diamond$ for its functor of points on perfectoid spaces over $C$, which is a diamond over $\Spd C$ by \cite[Lemma 15.6]{scholze2017etale}. We say that a diamond $D$ over $\Spd C$ is perfectoid if it is representable by a perfectoid space over ${\Spa C}$, i.e. if $D/\Spd C \cong P^\diamond/\Spd C$ for some perfectoid space $P/{\Spa C}$ (which, by Yoneda, is determined up to unique isomorphism). We recall also that any diamond $D$ admits an underlying topological space $|D|$, and that the open subdiamonds of $D$ are in natural bijection with the open subspaces of $|D|$. 

For any topological space $T$, we will write $\underline{T}$ as the associated $v$-sheaf given by $\underline{T}(X)=C^0(|X|,T)$ for any perfectoid space over $C$. By \cite[Lemma 10.13]{scholze2017etale}, for any locally profinite group $G$, the associated $v$-sheaf $\underline{G}$ over $\Spd C$ is represented by a perfectoid space.

For $X/C$ a rigid analytic space, there is a natural functor from the category $X_\proet$ to the category of diamonds over $X^\diamond$ sending $U=(U_i)_{i \in I} \in X_\proet$ to $U^\diamond/X^\diamond:=\varprojlim_{i \in I} U_i^\diamond/X^\diamond$. In particular, it is an easy consequence of the definitions that if $U \in X_\proet$ is perfectoid, then $U^\diamond/\Spd C$ is perfectoid.

\subsection{Geometric Sen theory}
Suppose $X/C$ is a smooth rigid analytic space and $x \in X(C)$. For $K$ a compact $p$-adic Lie group and $\rho: \pi_1(X,x) \rightarrow K$ continuous representation, the associated $K$-torsor is the object of $X_\profet$ associated to the continuous action of $\pi_1(X,x)$ on $K$ by left multiplication through $\rho$. It represents a right torsor for the topological constant sheaf $\ul{K}$ on $X_\proet$. Conversely, every such torsor arises from such a morphism. The results of \cite{JuangeometricSen} associate to any $\ul{K}$-torsor $X_\infty \in X_\proet$ a morphism of sheaves on $X_\proet$
\[ \kappa_{X_\infty}: T_{X/C} \rightarrow (\Lie K)_\rho \otimes_{\ul{\mathbb{Q}}_p} {\widehat{\mathcal{O}}_X}(-1), \]
where here $(\Lie K)_\rho$ denotes the $\ul{\mathbb{Q}_p}$-local system on $X$ associated to the adjoint action of $\rho$. The construction of the geometric Sen morphism is functorial in pullback of torsors along maps between smooth rigid analytic varieties and in push-out along maps between $p$-adic Lie groups. 

\subsection{Raynaud uniformization}

Let $A$ be an abeloid variety over a non-archimedean field $K$ with ring of integers $\mathscr{O}_K$, that is, a rigid analytic group variety which is proper, smooth, and connected.  Let
\begin{equation}\label{eq.raynaud-unif-extension}	0\rightarrow T\rightarrow E\rightarrow B\rightarrow 0	\end{equation}
be an extension of a formal abelian scheme $B$ over $\mathcal{O}_K$ by a split formal torus, and let $T^{\rig}$, $E^{\rig}$, and $B^{\rig}$ denote the rigid analytic generic fibers of $T$, $E$, and $B$, respectively.  If $T^{\rig}$ is isomorphic to a $d$-fold product of copies of the unit circle, there is a natural map $T^{\rig}\hookrightarrow T_K\cong \mathbb{G}_m^{\mathrm{ad},d}$; pushing out by this this map gives us an extension
\[	0\rightarrow T_K\rightarrow E_K\rightarrow B\rightarrow 0	\]

We say that a subgroup $M\subset E_K(K)$ is a \emph{lattice} if it is a discrete subgroup isomorphic to $\mathbb{Z}^{\oplus d}$.  If $M$ is a lattice, the quotient $E_K/M$ exists as a smooth proper rigid analytic space, by \cite[Proposition 6.1.4]{lutkebohmert2016}.  We say that $A$ admits a Raynaud uniformization if there exist an extension $E$ as in \cref{eq.raynaud-unif-extension} and lattice $M\subset E_K$ such that $A\cong E/M$.

\begin{theorem}{\cite[Theorem 7.6.4]{lutkebohmert2016}}
    	After replacing $K$ with a finite extension, $A$ admits a Raynaud uniformization.
\end{theorem}

\begin{remark}
	Note that this amounts to a semi-stable reduction theorem for abeloid varieties.
\end{remark}

\section{$p$-divisible rigid analytic groups}\label{s.p-divisible-rigid-groups}

In this section we recall the definition and basic structure results on $p$-divisible rigid analytic groups from \cite{Fargues-RigAnPDiv}, then use these to prove the cases of \cref{theorem.cover-classification} and \cref{theorem.geometric-sen} where $G$ is a $p$-divisible rigid analytic group over $C$. 

\subsection{Definition and classification}

The following is \cite[D\'efinition 2]{Fargues-RigAnPDiv} (see also \cite[D\'efinition 1.1]{Fargues-RigAnPDiv2}), except that we view $p$-divisible rigid analytic groups as adic spaces rather than Berkovich spaces. This does not effect the fundamental results or their proofs; see also \cite[Definition 5.2.2]{scholze2012moduli}. 

\begin{definition}[{\cite[D\'efinition 2]{Fargues-RigAnPDiv}}]
A $p$-divisible rigid analytic group over $G$ is an adic space locally topologically of finite type over $\Spa C$ equipped with a commutative group structure such that the multiplication by $p$ homomorphism $[p]_G: G \rightarrow G$ is
\begin{enumerate}
\item topologically nilpotent, in the sense that if $U$ and $V$ are two affinoid neighborhoods of the identity element then, for $n\gg 0$, $[p]_G^n|_U$ factors through $V$, and
\item finite and surjective. 
\end{enumerate}
\end{definition}

The main tool we will use is the classification result \cite[Th\'{e}or\'{e}me 3.3]{Fargues-RigAnPDiv}, which gives a constructive equivalence between $p$-divisible rigid analytic groups over $\Spa C$ and a linear algebra category. In the rest of the subsection we recall from \cite{Fargues-RigAnPDiv} this equivalence. 

If $G/C$ is a $p$-divisible rigid analytic group, we write $\Lie G$ for the tangent space at the identity, viewed as a vector group. There is a surjective \'{e}tale logarithm homomorphism $\log_G: G \rightarrow \Lie G$ fitting in an exact sequence
\[ 0 \rightarrow G[p^\infty] \rightarrow G \xrightarrow{\log_G} \Lie G \rightarrow 0.\]
Note that the vector group $\Lie G$ and the \'{e}tale group $G[p^\infty]$ are both themselves $p$-divisible rigid analytic groups, so that $G$ is an extension of a vector group by an \'{e}tale group in this category.   

\begin{example}
We write $\widehat{\mathbb{G}}^\rig_m$ for the open unit ball around $1$ and $\mathbb{G}_a^{\mathrm{ad}}$ for the rigid analytic affine line treated as a vector group. The group $\widehat{\mathbb{G}}^\rig_m$ is a $p$-divisible rigid analytic group with logarithm sequence
\[ 0 \rightarrow \mu_{p^\infty} \rightarrow \widehat{\mathbb{G}}^\rig_m \xrightarrow{\log} \mathbb{G}_a^{\mathrm{ad}} \rightarrow 0 \]
where here $\mu_{p^\infty}$ denotes the $p$-power roots of unity in $C$ and $\log$ is defined by the usual power series
\[ \log(1+t)=\sum_{n \geq 1} \frac{(-1)^n}{n}t^n.\]
\end{example}

\begin{theorem}[{\cite[Th\'{e}or\`{e}me 3.3]{Fargues-RigAnPDiv}}]
\label{theorem.fargues-rig-an-dual-HT}Let $G/C$ be a $p$-divisible rigid analytic group. Then there exists a canonical $C$-linear morphism, functorial in $G$, the dual Hodge-Tate homomorphism 
\[ {f}:\Lie G \rightarrow \Hom( \mathbb{Q}_p/\mathbb{Z}_p(1), G[p^\infty])\otimes C\simeq  T_p G \otimes C(-1) \]
and a commutative diagram, functorial in $G$,
\begin{equation}
\label{eq.pullback-p-div-group}
\begin{tikzcd}
 0 \ar[r]  & {G[p^\infty]}\ar[r] \ar[d,"\simeq"] & G \arrow[dr, phantom, "\usebox\pullback" , very near start, color=black] \ar[r,"\log_G"] \ar[d] &  {\Lie G}\ar[r] \ar[d,"{f}"] & 0 \\
0\ar[r] & T_p G(-1) \otimes \mu_{p^\infty} \ar[r] & {T_p G(-1)\otimes \widehat{\mathbb{G}}^{\mathrm{rig}}_m} \ar[r,"{\Id\otimes \log}"] & {T_p G(-1)\otimes \mathbb{G}_a^{\mathrm{ad}}} \ar[r] & 0
\end{tikzcd}
\end{equation}
realizing $G$ as a pullback. The functor $G\mapsto (\Lie G,T_pG,{f})$
\[\{ \text{$p$-divisible rigid analytic groups over $C$}\}\rightarrow \substack{\{ (V,\Lambda,f) \ |\ \text{$V$ a finite $C$-vector group, $\Lambda$ a finite free $\mathbb{Z}_p$-module,}\\ \text{ $f:V\to \Lambda \otimes C(-1)$ a $C$-linear morphism}\}}\]
is an equivalence of categories, with inverse the functor sending a triple $ (V,\Lambda,f) $, to the pullback of 
\[\begin{tikzcd}
& V \ar[d,"f"] \\
\Lambda(-1)\otimes \widehat{\mathbb{G}}^{\mathrm{rig}}_m \ar[r,"{\Id\otimes \log}"] & {\Lambda(-1)\otimes \mathbb{G}_a^{\mathrm{ad}}}.
\end{tikzcd}\]
\end{theorem}

\begin{example}
It follows that an arbitrary $p$-divisible rigid analytic group over $C$ can be decomposed as a product of a vector group and a $p$-divisible rigid analytic group such that $\Lie G \rightarrow T_p G \otimes C(-1)$ is injective. The latter subcategory admits another important description by \cite[Theorem B]{scholze2012moduli}: For $\mathfrak{G}/\mathcal{O}_C$ a $p$-divisible group, the rigid analytic generic fiber $\mathfrak{G} \mapsto \mathfrak{G}^\rig$ is an equivalence between $p$-divisible groups over $\mathcal{O}_C$ and the full subcategory of $p$-divisible rigid analytic groups $G/C$ such that the dual Hodge-Tate map map $\Lie G \rightarrow T_p G \otimes C(-1)$ is injective. The triple corresponding to $\mathfrak{G}^\rig$ is $(T_p \mathfrak{G}(C), \Lie \mathfrak{G}\otimes C, \mathfrak{\alpha}^\vee)$, where $\alpha^\vee$ is the dual Hodge-Tate map of \cite{tate1967p-divisible}. In particular, the dual Hodge-Tate map is the inclusion of the subspace in the Hodge-Tate filtration on $T_p G \otimes C(-1)$. 
\end{example}

\subsection{The universal cover}

\begin{definition}
\label{def.universal cover of p-divisible rigid group}
   For $G/C$ a $p$-divisible rigid analytic group, we consider the tower\footnote{Since $\widetilde{G}$ it is not typically the universal profinite \'{e}tale cover of $G$, we avoid the terminology universal cover for this object suggested by \cite{scholze2012moduli}.};  $\widetilde{G}\coloneqq(G)_{n \in \mathbb{Z}_{\geq 0}} \in G_\proet$ where the transition map between level $n$ and level $m$, $n \geq m$, is $[p^{n-m}]_G.$ It is a profinite \'{e}tale $T_p G$-cover of $G$. The associated diamond is   
\[ \widetilde{G}^\diamond \coloneqq \lim \left( G^\diamond \xleftarrow{[p]_{G^\diamond}} G^\diamond  \xleftarrow{[p]_{G^\diamond}} G^\diamond \xleftarrow{[p]_{G^\diamond}} \ldots \right)\]
and projection to the first coordinate $\widetilde{G}^\diamond \rightarrow G^\diamond$ is a profinite \'{e}tale $T_pG$-cover of $G^\diamond$. 
\end{definition}

\begin{remark}
For any integer $n$ with $(n,p)=1$, multiplication by $n$ is invertible on $G$. Thus the canonical projection map from the limit over all multiplication by $n$ covers to $\widetilde{G}$, the limit over all multiplication by $p^k$ covers, is an isomorphism. Thus this does not conflict with the notation we have used in the introduction. In particular, we have $T_{\widehat{\mathbb{Z}}}G=T_pG$ by the obvious map. 
\end{remark}

\begin{notation}
    For $G/C$ a $p$-divisible rigid analytic group, we write $G_H$ for the object associated to the induced map $\pi_1(G,e) \rightarrow T_p G/H$ for $H$ a closed subgroup of $T_p G$,. Concretely, it is the tower of finite \'{e}tale covers $(G_U)_U$ where $U$ runs over all open subgroups of $T_p G$ containing $H$.
\end{notation}

The assignment $G \mapsto \widetilde{G}^\diamond$ is a functor from $p$-divisible rigid analytic groups over $C$ to $\mathbb{Q}_p$-vector space objects in the category of diamonds. Applying this functor to the presentation of $G$ as a pullback \cref{eq.pullback-p-div-group}, we obtain a presentation of $\widetilde{G}^\diamond$ as a pullback

\[
\begin{tikzcd}
 \Lambda\otimes \ul{\mathbb{Q}_p} \ar[r] \ar[d,"\simeq"] & {\widetilde{G}^\diamond} \arrow[dr, phantom, "\usebox\pullback" , very near start, color=black] \ar[r] \ar[d] &  {\Lie G}^\diamond\ar[d,"f^{\diamond}"] \\
 \Lambda(-1) \otimes \ul{\mathbb{Q}_p(1)} \ar[r] & \Lambda(-1)\otimes \widetilde{\widehat{\mathbb{G}}_m^{\mathrm{rig}}}^\diamond \ar[r] & {\Lambda(-1)\otimes \mathbb{G}_a^{\mathrm{ad},\diamond}}
\end{tikzcd}
\]

\subsection{Proof of \cref{theorem.cover-classification} for $p$-divisible rigid analytic groups.}

The following is the case of \cref{theorem.cover-classification} when $G$ is a $p$-divisible rigid analytic group. 
\begin{proposition}\label{prop.cover-classification-p-divisible-group}Let $G/C$ be a a $p$-divisible rigid analytic group with dual Hodge-Tate map ${f}: \Lie G \rightarrow T_p G \otimes C(-1)$. 
\begin{enumerate}
\item If ${f}^{-1}(H\otimes C(-1))=\{0\}$, then $G_H$ is a a perfectoid object of $X_\proet$.
\item Moreover, if ${f}^{-1}(H \otimes C(-1)) \neq \{0\}$, then no open subdiamond of $G_H^\diamond$ is perfectoid. 
\end{enumerate}
\end{proposition}
\begin{proof}
We write $\Lambda:=T_p G$. 

We first treat claim (1).  Thus, let $H \leq \Lambda$ be a subgroup such that that $({f})^{-1}(H\otimes C(-1))=0$. We want to show $\widetilde{G}/H$ is perfectoid. Since $({f})^{-1}(H\otimes C(-1))=0$, we can find a subgroup $H'$ containing $H$ such that the composition 
\[ \Lie G \xrightarrow{f} \Lambda \otimes C(-1) \to  (\Lambda/H') \otimes C(-1) \]
is an isomorphism. Then, applying almost purity via \cref{lemma.pro-etale-cover-perf-is-perf} to the cover $G_H \rightarrow G_{H'}$, we find it suffices to show that $G_{H'}$ is perfectoid. Replacing $H$ with $H'$, we may assume below that the composition of  $\Lie G \xrightarrow{f} \Lambda \otimes C(-1) \to  \Lambda/H \otimes C(-1)$ is an isomorphism. 

Now let $\overline{\Lambda}=\Lambda/H$ and $q: \Lambda(-1) \rightarrow \overline{\Lambda}(-1)$. We obtain a natural homomorphism $\pi: G \rightarrow \overline{\Lambda}(-1) \otimes_{\mathbb{Z}_p} \widehat{\mathbb{G}}_m^\rig$ by composing the middle vertical arrow of \cref{eq.pullback-p-div-group} with $q \otimes \Id$. The derivative $d\pi$ at identity in $G$ is given by the composition of ${f}$ with the quotient map $q_C: \Lambda(-1) \otimes C \rightarrow \overline{\Lambda}(-1) \otimes C$. By our assumption above, this is an isomorphism, thus $\pi$ is \'{e}tale. 

Now, we claim that 
\begin{equation}\label{eq.pullback-rigan-proof} G_H \cong \pi^* \left( \widetilde{ \overline{\Lambda}(-1) \otimes \widehat{\mathbb{G}}_m^{\mathrm{rig}}} \right), \end{equation}
where the pullback is of objects in the pro-\'{e}tale site. Indeed, if we note that $\pi$ induces the quotient map $q: \Lambda \rightarrow \overline{\Lambda}$ on Tate modules, then we find that both sides are the cover associated to 
\[ \pi_1(G,e) \rightarrow \Lambda \rightarrow \overline{\Lambda}. \]
This implies that $G_H$ is perfectoid: $\widetilde{ \overline{\Lambda}(-1) \otimes \widehat{\mathbb{G}}_m^{\mathrm{rig}}}$ is evidently a perfectoid object of $(\overline{\Lambda}(-1) \otimes \widehat{\mathbb{G}}_m^{\mathrm{rig}})_\proet$, and the pullback of a perfectoid object in the pro-\'{e}tale site along an \'{e}tale map of rigid analytic varieties is again a perfectoid object (as can be seen, e.g., by writing the \'{e}tale map locally as a composition of rational localizations and finite \'{e}tale maps).

We now treat part (2). Thus, suppose that $H \leq \Lambda$ is a closed subgroup such that ${f}^{-1}(H\otimes C(-1)) \neq 0$.  Suppose, by way of contradiction, that there is an open $U \subseteq G_H^\diamond$ that is perfectoid. By translation, we may assume the identity $e_H: \Spd C \rightarrow G_H^\diamond$ factors through $U$, and by shrinking $U$ even further, we may assume it is affinoid perfectoid, $U=\Spa(R,R^+)^\diamond$ for $\Spa(R,R^+)$ affinoid perfectoid over $\Spa(C, \mathcal{O}_C)$. We will obtain a contradiction by exhibiting a non-zero continuous derivation from $R$ to $C$ (where $C$ is viewed as an $R$-module by $e_H^*$). By \cref{lem.no-derivation} below, no such continuous derivation can exist. 

To exhibit a nonzero derivation, we let $W'$ be the $C$-span of $f^{-1}(H\otimes C(-1))$ in $\Lie G$, a non-zero subspace, and consider the triple $(H, W', f|_{W'})$. This triple determines a $p$-divisible rigid analytic group $G'$, and there is a natural map $i: G' \rightarrow G$ determined by the inclusions $H \hookrightarrow \Lambda$ and $W' \hookrightarrow \Lie G$. We have an induced map of diamond universal covers $\widetilde{G'}^\diamond \rightarrow \widetilde{G}^\diamond$ and taking the quotient by $H$ yields
\[ s: G'^{,\diamond} = \widetilde{G'}^\diamond/H \rightarrow \widetilde{G}^\diamond/H=G_H^\diamond \]
fitting into a commutative diagram
\begin{equation}\label{eq.commutative-sections-pdivproof}\begin{tikzcd}
	& {G_H^\diamond} \\
	{(G')^\diamond} & {G^\diamond}
	\arrow["h", from=1-2, to=2-2]
	\arrow["s", from=2-1, to=1-2]
	\arrow["i"', from=2-1, to=2-2]
\end{tikzcd}\end{equation}
Since $\Lie G'=W' \neq 0$, we may take a non-zero $\partial \in \Lie G'$. Let $V=\Spa(A,A^+) \subseteq s^{-1}(U)$ be an affinoid open containing $e_{G'}$ (here we are using that $|G'|=|G'^\diamond|$). The map $V^\diamond \rightarrow U^\diamond$ induces a map on functions $R \rightarrow A$ sending $R^+$ to $A^+$ --- indeed, this follows from 
\[ R= \Hom_{\Spd C}(U, \mathbb{A}^{1,\diamond}), R^+=\Hom_{\Spd C}(U, \mathbb{D}^\diamond), A=\Hom_{\Spd C}(U, \mathbb{A}^{1,\diamond}), A^+=\Hom_{\Spd C}(U, \mathbb{D}^\diamond)\]
where here $\mathbb{D}$ is the disk $|t| \leq 1 \subseteq \mathbb{A}^1$, the first two equalities follow from Yoneda, and the second two equalities follow from the full-faithfulness of $X/C \rightarrow X^\diamond/\Spd C$ on smooth rigid spaces over $C$. 

Then, viewing $\partial$ as a continuous $C$-linear derivation $A \rightarrow C$, we may compose with $R \rightarrow A$ to obtain a continuous $C$-linear derivation $\partial_R: R \rightarrow C$ (the map $R \rightarrow A$ is continuous because it sends $R^+$ to $A^+$; here we use that both are rings of definition since both spaces are uniform). We claim $\partial_R$ is nonzero: indeed, by the commutativity of \cref{eq.commutative-sections-pdivproof}, $\partial_R (h^* b) = \partial (i^* b)$ for any $b \in \mathcal{O}(G^\diamond)=\mathcal{O}(G).$ But there are evidently global functions on $G$ such that $\partial i^*b \neq 0$ (e.g. the composition of $\log_G$  with an element of $W^*=(\Lie G)^*$ that does not vanish on $\partial \in W'\subseteq W$), so we conclude. 
\end{proof}

The following lemma is well known to experts and can be made considerably stronger; we give an elementary statement and argument that handles just the case we need. 
\begin{lemma}\label{lem.no-derivation}
Let $U=\Spa (R,R^+)$ be affinoid perfectoid over $\Spa C$, and let $u: {\Spa C} \rightarrow U$ be a point. Viewing $C$ as an $R$-module via $u^*$, any continuous $C$-linear derivation of $R$-modules $d: R \rightarrow C$ is identically zero.  
\end{lemma}
\begin{proof}
Suppose $d: R \rightarrow C$ is a non-zero $C$-linear derivation. It is surjective, thus, by the open mapping theorem for $C$-Banach algebras, $d(R^+)$ is open in $C$; multiplying $d$ by a power of $p$, we may thus assume $d(R^+) \subseteq \mathcal{O}_C$.  Now, consider $x \in R^+$. Since $R^+$ is perfectoid, there exists a $y \in R^+$ such that $x=y^p + p g$ for $g \in R^+$. Thus 
\[ d(x)=d(y^p + p g)=p(u^*(y^{p-1}) d(y) + d(g)).\]
Since $u^*(y^{p-1}) \in \mathcal{O}_C$, $d(y) \in \mathcal{O}_C$, and $d(g) \in \mathcal{O}_C$, we find $d(x)\in p \mathcal{O}_C$. Thus $d(R^+) \subseteq p \mathcal{O}_C$. Repeating the argument, we find $d(R^+) \subseteq p^n \mathcal{O}_C$ for all $n \geq 0$, thus $d(R^+) \subseteq \bigcap_n p^n \mathcal{O}_C = \{0\}$. But then we also have $d(R)=\{0\}$, contradicting the assumption that $d$ was a nonzero derivation.  
\end{proof}

\subsection{Computation of the geometric Sen morphism}

\begin{lemma}\label{lemma.geometric-sen-p-divisible-universal-cover} Let $G/C$ be a $p$-divisible rigid analytic group. The geometric Sen morphism for the $T_p G$-torsor $\widetilde{G}$ is the constant extension of the dual Hodge-Tate map 
\begin{equation}
\label{eq.constant-inclusion}
f \otimes_C \Id:  T_{G/C}\simeq \Lie G \otimes_C \widehat{\mathcal{O}}_G  \rightarrow (T_p G \otimes C(-1))  \otimes_C \widehat{\mathcal{O}}_G =T_p G \otimes \widehat{\mathcal{O}}_G(-1). \end{equation}
\end{lemma}
\begin{proof}
The universal cover of $\widetilde{G}$ is pulled back from the universal cover of $T_p G(-1) \otimes \widehat{\mathbb{G}}_m^\rig$ by the middle vertical arrow of  \cref{eq.pullback-p-div-group}. Thus, applying functoriality of the geometric Sen morphism under pullback (\cite[Theorem 1.0.3-(5)]{JuangeometricSen}) and observing that the derivative of this arrow $G \rightarrow T_p G(-1) \otimes \widehat{\mathbb{G}}_m^\rig$ is the constant extension of $f$, we are reduced to the case of $T_p G(-1) \otimes \widehat{\mathbb{G}}_m^\rig$. 
Then, if we fix a basis of $T_p G(-1)$, since $\widetilde{G}$ is the product of the the pullbacks of the universal cover $\widehat{\mathbb{G}}^\mathrm{rig}_m$ along the corresponding projection maps, we are reduced by functoriality under pullbacks (\cite[Theorem 1.0.3-(5)]{JuangeometricSen}) and compatibility with products (as contained in (\cite[Theorem 1.0.4]{JuangeometricSen}) to the case $G=\widehat{\mathbb{G}}^\mathrm{rig}_m$. 
By functoriality under pullbacks (\cite[Theorem 1.0.3-(5)]{JuangeometricSen}), it suffices to compute the geometric Sen morphism for the standard $\mathbb{Z}_p(1)$-torsor over the torus 
\[ \Spa C\langle t^{\pm 1/p^\infty}\rangle =: \mathbb{T}_\infty \rightarrow  \mathbb{T}:=\Spa C\langle t^{\pm 1} \rangle,\]
since the restriction of $\mathbb{T}_\infty$ to $\widehat{\mathbb{G}}^\mathrm{rig}_m\subseteq \mathbb{T}$ is $\widetilde{\widehat{\mathbb{G}}^{\mathrm{rig}}_m}$. By \cite[Example 3.1.6]{pan2022locally}, the Sen morphism in this case sends $t\partial_t$ to the element $1 \in (\Lie \mathbb{Z}_p(1) \otimes \widehat{\mathcal{O}}(-1))(\mathbb{T})=\widehat{\mathcal{O}}(\mathbb{T}).$ The restriction of this to $\widehat{\mathbb{G}}_m^\rig$ is the constant extension of the dual Hodge-Tate map for $\widehat{\mathbb{G}}_m^\rig$, as desired. 
\end{proof}

\begin{proposition}\label{prop.geometric-sen-p-divisible}
Let $G/C$ be a $p$-divisible rigid analytic group. Let $K$ be a compact $p$-adic lie group. Suppose $\rho: T_{p} G \rightarrow K $ is a continuous homomorphism. Then, for the associated $K$-torsor $G_\infty$, the geometric Sen morphism is the constant composition of $d\rho$ with the dual Hodge-Tate map $f: \Lie G \rightarrow T_p G\otimes C(-1)$, i.e. 
\[ \kappa_{G_\infty} = \left( (d\rho \otimes C) \circ f \right) \otimes_C \widehat{\mathcal{O}}_G: T_{G/C} \rightarrow \Lie (\mathrm{Im}(\rho)) \otimes_{\mathbb{Q}_p} \widehat{\mathcal{O}}(-1) \subseteq (\Lie K)_\rho \otimes_{\ul{\mathbb{Q}_p}} \widehat{\mathcal{O}}(-1). \]  
Moreover, $G_\infty$ is perfectoid if and only if this is injective at each point, i.e. if and only if
\[ (d\rho_p \otimes C) \circ f \textrm { is injective.}\] 
In fact, if this map is not injective, then no open subdiamond of $G_\infty^\diamond$ is perfectoid. 
\end{proposition}
\begin{proof}
The computation of the geometric Sen morphism follows from \cref{lemma.geometric-sen-p-divisible-universal-cover} and functoriality of the geometric Sen morphism under push-out along maps of $p$-adic Lie groups by \cite[Theorem 1.0.4]{JuangeometricSen}. The remaining claims then follow from \cref{prop.cover-classification-p-divisible-group}: Let $H=\ker \rho$ and let $\overline{\rho}: T_p G/H \rightarrow K$ be the induced map. Then  
\[ G_\infty = \widetilde{G}/H \wedge^{\overline{\rho}} K = \widetilde{G}/H \times T\]
where $T \leq K$ is the image of any continuous section of $K \rightarrow \mathrm{Im}(\overline{\rho})\backslash K$ (such a section exists, e.g. because this quotient map is a submersion of $p$-adic manifolds). Thus, if $G_H$ is perfectoid so is $G_\infty$, and conversely if there is a non-empty perfectoid open $V \subseteq G_\infty^\diamond$, then intersecting with a suitable fiber over $T$, we find a non-empty perfectoid open of $\widetilde{G}^\diamond/H$.   

\end{proof}

\section{Semi-abeloid varieties}\label{s.semi-abeloid-varieties}

In this section we establish our results on semi-abeloid varieties, including  the cases of \cref{theorem.cover-classification} and \cref{theorem.geometric-sen} where $G$ is a semi-abelian variety. The non-perfectoidness results and a weak version of the perfectoidness result are established in \cref{ss.semi-abeloid-results-using-p-div-groups} as an application of the results for $p$-divisible rigid analytic groups established in \cref{s.p-divisible-rigid-groups}. The perfectoidness result for semi-abelian varieties, which requires a different argument, is established in \cref{ss.perfectoid-covers}. 

\subsection{Definitions and the associated $p$-divisible rigid analytic group}
Recall that an abeloid variety is a rigid analytic group variety which is proper, smooth, and connected, and a semi-abeloid variety is a rigid analytic group variety that is an extension of an abeloid variety by a torus.
\begin{definition}\label{def.prof-etale-Etilde-cover}
    For $E/C$ a semi-abeloid variety, we consider the tower $\widetilde{E}\coloneqq(E)_n \in E_\proet$ where the transition maps from the $n$th term to the $m$th term, $m|n$ is multiplication $n/m$. It is a profinite \'{e}tale $T_{\widehat{\mathbb{Z}}} E$-cover of $E$. The associated diamond is  $\widetilde{E}^\diamond \coloneqq \lim_{[n]_{E^\diamond}} E^\diamond$
and projection to the first coordinate $\widetilde{E}^\diamond \rightarrow E^\diamond$ is a profinite \'{e}tale $T_{\widehat{\mathbb{Z}}} E$-cover of $E^\diamond$.
\end{definition}

\begin{remark}
    It is well-known that the profinite \'etale fundamental group of an abeloid variety $E$ is $T_{\widehat{\mathbb{Z}}}E$. Due to the lack of reference outside of the abelian case, we give a proof in \cref{appedix.abeloid-fundamental-group}. However, as noted in the introduction, $\widetilde{E}$ is the universal cover in the abeloid case but may not be the universal cover for a more general semi-abeloid. 

\end{remark}

\begin{construction}
\label{cons.p-divisible group attached to abeloid}
For $E/C$ a semi-abeloid variety, we define the $p$-divisible rigid analytic group $G$ of $E$ as 
\[ G\coloneqq\lim_{U \subseteq E \textrm{ an open subgroup}} \bigcup_{n \geq 0} [p]_{E}^{-n}(U). \]
The limit exists: any sufficiently small open subgroup is isomorphic by exponentiation to an open subgroup of the vector group $\Lie E$, and the the projection of the limit to the corresponding term is an isomorphism. In particular, $G$ is an open rigid analytic subgroup of $E$.
\end{construction}

We may view $\widetilde{G}^\diamond$ as a subgroup of $\widetilde{E}^\diamond$. It is not an open subgroup, but we define $\widetilde{G}^\diamond_{\widehat{\mathbb{Z}}}$ to be the open preimage of $G^\diamond$ in $\widetilde{E}^\diamond$.

\begin{lemma}
    Suppose $E/C$ is a semi-abeloid variety and let $G$ be its associated $p$-divisible rigid analytic group. Then there is an natural isomorphism via multiplication 
    $$\widetilde{G}^\diamond \times \ul{T_{\widehat{\mathbb{Z}}^{(p)}}E} \xrightarrow{\simeq}\widetilde{G}^\diamond_{\widehat{\mathbb{Z}}} $$
    where $T_{\widehat{\mathbb{Z}}^{(p)}} E= \varprojlim_{(n,p)=1} E[n](C)$.
\end{lemma}
\begin{proof}
    Note that there is an natural surjection $\widetilde{G}^\diamond_{\widehat{\mathbb{Z}}} \to \widetilde{G}^\diamond $ by projecting onto the pro-$p$ part. We claim that the kernel is given by $\ul{T_{\widehat{\mathbb{Z}}^{(p)}}E}$, i.e. there is an split short exact sequence
    $$0\to \ul{T_{\widehat{\mathbb{Z}}^{(p)}}E} \to \widetilde{G}^\diamond_{\widehat{\mathbb{Z}}} \to \widetilde{G}^\diamond \to 0 .$$
    Indeed, when $(n,p)=1$, since multiplication by $n$ on $G$ is an isomorphism the $E[n]^\diamond$-torsor $[n]_{E^\diamond}^*G^\diamond \to G^\diamond$ splits. Hence $[n]_{E^\diamond}^*G^\diamond \simeq G^\diamond \times E[n]^\diamond$. Identifying $E[n]^\diamond$ with $\ul{E[n](C)}$ and taking the inverse limit, we get the desired split short exact sequence. 
\end{proof}

\subsection{Results deduced using the $p$-divisible rigid analytic group}\label{ss.semi-abeloid-results-using-p-div-groups}

Using $\widetilde{G}^\diamond_{\widehat{\mathbb{Z}}}$, we now deduce versions of \cref{theorem.cover-classification} and \cref{theorem.geometric-sen} for semi-abeloid varieties from the corresponding results for $p$-divisible rigid analytic groups, \cref{prop.cover-classification-p-divisible-group} and \cref{prop.geometric-sen-p-divisible}, but where the perfectoidness implication is weakened.

\begin{proposition}\label{prop.cover-classification-semi-abeloid}Suppose $E/C$ is a semi-abeloid variety, let $G$ be its associated $p$-divisible rigid analytic group, and let $H=\prod H_\ell \leq T_{\widehat{\mathbb{Z}}}E$ be a closed subgroup. For $f$ the dual Hodge-Tate map, if $f^{-1}(H_p \otimes C(-1))\neq 0$, then no open sub-diamond of $E_H^\diamond$  is represented by a perfectoid space. If $f^{-1}(H_p \otimes C(-1))=\{0\}$, then for $U=(G \times \{ e\}) \cdot \Delta_E$, an open neighborhoood of the diagonal $\Delta_E \le E \times E$, projection to the second factor $(E_H^\diamond \times E_H^\diamond)|_{U^\diamond} \rightarrow E_H^\diamond$ is  representable in perfectoid spaces. 
\end{proposition}
\begin{proof}
Suppose first that $f^{-1}(H_p\otimes C(-1))\neq 0$. Suppose $V \subseteq E_H^\diamond=\widetilde{E}^\diamond/H$ is a non-empty perfectoid open. If there is a point $u: \Spa(C) \rightarrow V$, then $V \cap (\widetilde{G}^\diamond/H_p) \cdot \{u\}$ is identified with a non-empty open subset of $\widetilde{G}^\diamond$ which is perfectoid as a Zariski closed subset of the open perfectoid $V \cap ((\widetilde{G}^\diamond_{\widehat{\mathbb{Z}}}/H) \cdot \{u\})$\footnote{Here we implicitly used that if $\Spd(R,R^+)$ is perfectoid, then for any finite morphism $(R,R^+)\to (S,S^+)$ the associated diamond $\Spd(S,S^+)$ is perfectoid, c.f. the proof of \cref{corollary.variety with gg diff}.}. This contradicts \cref{prop.cover-classification-p-divisible-group}. In general, there is always a geometric rank-1 point $\Spa(C') \rightarrow V$, for some algebraically closed non-archimedean extension $C'/C$, thus by extending scalars from $C$ to $C'$ we reduce to case where there is such a point $u$ of $V$.   

Suppose now that $f^{-1}(H_p\otimes C(-1))=0$. Then $\widetilde{G}^\diamond/H_p$ is perfectoid by \cref{prop.cover-classification-p-divisible-group}, and it follows that 
\[ \widetilde{G}^\diamond_{\widehat{\mathbb{Z}}}/H=(\widetilde{G}^\diamond/H_p) \times (\underline{T_{\widehat{\mathbb{Z}}^{(p)}}E }/ \prod_{\ell\neq p}H_\ell) \]
is perfectoid as it is a product of two perfectoid spaces. The open neighborhood $U= (G\times \{ e\}) \cdot \Delta_E$ is the pre-image of $G \leq E$ under the map $E \times E \rightarrow E, (x_1, x_2)\mapsto x_1 x_2^{-1}$. Thus, for any perfectoid space $T$ and $x: T \rightarrow \widetilde{E}^\diamond/H$, we obtain an isomorphism
\[ \widetilde{G}^\diamond_{\widehat{\mathbb{Z}}}/H \times T \xrightarrow{(g, t) \mapsto (g \cdot x(t), x(t), t)} (\widetilde{E}^\diamond/H \times \widetilde{E}^\diamond/H)|_{U^\diamond} \times_{\widetilde{E}^\diamond/H} T \subseteq \widetilde{E}^\diamond/H \times T  \] 
Indeed, a point $(a,t)$ is in the codomain if and only if $\pi(a)\pi(x(t))^{-1}=\pi(ax(t)^{-1}) \in G^\diamond \leq E^\diamond$ under the map $\pi: \widetilde{E}^\diamond/H \rightarrow E^\diamond$, and $\widetilde{G}_{\widehat{\mathbb{Z}}}^\diamond/H \leq \widetilde{E}^\diamond/H$ is precisely the pre-image of $G$ under $\pi$. 
\end{proof}

\begin{proposition}\label{prop.geometric-sen-semi-abeloid}
Let $E/C$ be a semi-abeloid variety. Suppose $K$ is a compact $p$-adic lie group and $\rho: T_{\widehat{\mathbb{Z}}} E \rightarrow K$ is a continuous homomorphism and let $E_\infty$ be the associated $K$-torsor. Let $\rho_p=\rho|_{T_p E}$. The geometric Sen morphism for $E_\infty$ is the constant composition of $d\rho_p$  with the dual Hodge-Tate map $f: \Lie E \rightarrow T_p E\otimes C(-1)$  i.e.
\[ \kappa_{E_\infty} = \left( (d\rho_p \otimes C) \circ f \right) \otimes_C \widehat{\mathcal{O}}: T_{E/C} \rightarrow \Lie (\mathrm{Im}(\rho_p)) \otimes_{\mathbb{Q}_p} \widehat{\mathcal{O}}(-1) \subseteq (\Lie K)_\rho \otimes_{\ul{\mathbb{Q}_p}} \widehat{\mathcal{O}}(-1)\]  
If $(d\rho_p \otimes C) \circ f$ is not injective, then no open subdiamond of $E_\infty^\diamond$ is perfectoid. If it is injective, then, for $G$ the $p$-divisible rigid analytic group of $E$ and $U=(G \times \{ e\}) \cdot \Delta_E$, an open neighborhood of the diagonal $\Delta_E \leq E \times E$, projection to the second factor $(E_\infty^\diamond \times E_\infty^\diamond)|_{U^\diamond} \rightarrow E_\infty^\diamond$ is representable in perfectoid spaces. 

\end{proposition}
\begin{proof}
After the computation of the geometric Sen morphism, the rest of the result follows from \cref{prop.cover-classification-semi-abeloid} by an argument similar to the deduction of the second part \cref{prop.geometric-sen-p-divisible} from \cref{prop.cover-classification-p-divisible-group}. For the computation of the geometric Sen morphism, similar to the proof of \cref{prop.geometric-sen-p-divisible}, it suffices to compute the geometric Sen morphism for covers $\widetilde{E}/H^{(p)} \rightarrow E$ where $H^{(p)}$ is an open subgroup of $T_{\widehat{\mathbb{Z}}^{(p)}} E$ (for a given $\rho$, the relevant subgroup $H^{(p)}$ is the kernel of the restriction of $\rho$ to $T_{\widehat{\mathbb{Z}}^{(p)}} E$). For any rank one point, $x: \Spa(C') \rightarrow E$, multiplication of $x$ by $G_{C'}$ gives an open immersion $G_{C'} \rightarrow E_{C'}$ sending the identity to $x$ and extending to an isomorphism of $(\widetilde{G}_{\widehat{\mathbb{Z}}}/H^{(p)})_{C'}$ with the restriction of $\widetilde{E}_{C'}/H^{(p)}$ along this open immersion. Thus we deduce the computation of the geometric Sen morphism at this point from \cref{prop.geometric-sen-p-divisible}, and the result follows since this morphism is uniquely determined by its values at rank one points. 
\end{proof}

\subsection{Perfectoid covers}\label{ss.perfectoid-covers}

In this section we establish 

\begin{proposition}\label{prop.perfectoid-covers-abelian}Let $E/C$ be a semi-abelian variety. Then, for $H=\prod H_\ell \leq T_{\widehat{\mathbb{Z}}}E$ a closed subgroup, $E_H\in E_\proet$ is perfectoid if $f^{-1}(H_p\otimes C(-1))=\{0\}$ for $f: \Lie E \rightarrow T_p E \otimes C(-1)$ the dual Hodge-Tate map. 
\end{proposition}

Combined with the computation of the geometric Sen morphism \cref{prop.geometric-sen-semi-abeloid}, we obtain also
\begin{corollary}\label{cor.geometric-sen-perfectoid-semi-abelian}
Let $E/C$ be a semi-abelian variety.
Suppose $K$ is a compact $p$-adic lie group and $\rho: T_{\widehat{\mathbb{Z}}} E \rightarrow K$ is a continuous homomorphism, and let $\rho_p=\rho|_{T_p E}$. If the associated geometric Sen morphism is injective at each geometric point (equivalently, $d\rho_p \otimes C \circ f$ is injective), then the associated $K$-torsor $E_\infty/E$ is perfectoid. 
\end{corollary}
\begin{proof} 
We argue from \cref{prop.perfectoid-covers-abelian} as in the proof of \cref{prop.geometric-sen-p-divisible}. 
\end{proof}

Before beginning the proof of \cref{prop.perfectoid-covers-abelian}, we prove some lemmas that will allow us to reduce to the case of an abelian variety of good reduction. 

\begin{lemma}\label{lemma.semi-abelioid-to-abeloid}
Let $E$ be a semi-abeloid variety over $C$, written as an extension
\[ 0 \rightarrow T \rightarrow E \rightarrow A \rightarrow 0 \]
where $T\cong \mathbb{G}_m^{\mathrm{ad},k}$ is a torus and $A$ is an abeloid variety. 
Suppose $H \leq T_p E$ is a summand such that $H \otimes C \cap \Lie E(1)=\{0\}$. Then $H \hookrightarrow T_p A$ with $H\otimes C \cap \Lie A(1)=\{0\}$ and $E_H=\widetilde{E}/H$ is perfectoid if $A_H=\widetilde{A}/H$ is perfectoid. 
\end{lemma}
\begin{proof}
For the first part, suppose $H \otimes C \cap \Lie E(1) = \{0\}$. Since $T_p T \subseteq T_p T \otimes C = \Lie T(1) \subseteq \Lie E(1)$, we find $H \hookrightarrow T_p A$. The statement that $H \otimes C \cap \Lie A(1)=\{0\}$ viewing $H$ as a subgroup of $T_p A$ is equivalent to the statement that $H \otimes C \cap \Lie E(1) =\{0\}$, since $\Lie E(1)$ is the preimage in $T_p E \otimes C$ of $\Lie A(1)$ in $T_p A \otimes C$. 

For the second part, note that we have an exact sequence of pro-systems $0 \rightarrow \widetilde{T} \rightarrow \widetilde{E} \rightarrow \widetilde{A} \rightarrow 0$ and an induced quotient exact sequence
\[ 0 \rightarrow \widetilde{T} \rightarrow \widetilde{E}/H \rightarrow \widetilde{A}/H \rightarrow 0. \]
In particular, $\widetilde{E}/H$ is a $\widetilde{T}$-torsor over $\widetilde{A}/H$. Since $\widetilde{T}\cong (\mathcal{O}^{\flat,\times})^k$, if $\widetilde{A}/H$ is perfectoid then this torsor is analytically locally trivial on $\widetilde{A}/H$. Thus we may find a cover of $\widetilde{A}/H$ by open affinoid perfectoids $U$ such that $\widetilde{E}/H \times_{\widetilde{A}/H} U \cong \widetilde{T} \times_{\Spa C} U$, thus $\widetilde{E}/H$ is perfectoid by \cref{lem.product of perfedctoids in proet}. 
\end{proof}

\begin{lemma}\label{lemma.semi-abeloid-to-quotient}
Suppose $E/C$ is a semi-abeloid variety, $M \leq E(C)$ is a discrete free subgroup, and $A=E/M$. Suppose $H \leq T_p A$ is a closed subgroup such that $H \otimes C \cap \Lie A(1)=\{0\}$. Then, for $H' = H \cap T_p E$, $H' \otimes C \cap \Lie E(1)=\{0\}$, and if $\widetilde{E}/H'$ is perfectoid then $\widetilde{A}/H$ is perfectoid. 
\end{lemma}
\begin{proof}
Note that we have an exact sequence 
\[ 0 \rightarrow T_p E \rightarrow T_p A \rightarrow M \otimes \mathbb{Z}_p \rightarrow 0. \]
Since $\Lie A(1)=\Lie E(1) \subseteq T_p E \otimes C$, we always have $H' \otimes C \cap \Lie E(1) \subseteq H \otimes C \cap \Lie A(1)$; in particular if the latter is zero then so is the former. 

Now, assume that $H \cap \Lie A(1)=\{0\}$ and $\widetilde{E}/H'$ is perfectoid. We can cover $A$ by open subsets $U$ admitting a section $s: U \rightarrow E$; to finish the proof, it suffices to show that $\widetilde{A}/H|_U$ is perfectoid for any such $U$.  By \cref{lem.product of perfedctoids in proet}, $s^* (\widetilde{E}/H')$ is perfectoid. On the other hand, we have $\widetilde{A}/H|_U = s^*(\widetilde{E}/H') \wedge^{T_p E/H'} T_p A/H$ and we conclude that $\widetilde{A}/H|_U$ is perfectoid as in the proof of \cref{prop.geometric-sen-p-divisible}. 
\end{proof}

\begin{proof}[Proof of \cref{prop.perfectoid-covers-abelian}]
We first reduce to the case of an abelian variety with good reduction: \cref{lemma.semi-abelioid-to-abeloid} reduces us from the semi-abelian variety $E$ to its abelian quotient. Using the Raynaud uniformization of this abelian variety, \cref{lemma.semi-abeloid-to-quotient} reduces us to a semi-abelian variety whose abelian part has good reduction. Finally, by another application of \cref{lemma.semi-abelioid-to-abeloid}, we see it suffices to to assume that $E=A$ is an abelian variety with good reduction.

We next reduce to the case that $H_\ell= T_l A$ for all $\ell \neq p$ and $A$ is equipped with a principal polarization such that $H_p$ is a maximal isotropic subgroup of $T_p A$. 
To that end, first note that we may replace $A$ with $A'=A \times_{C} B$ for any good reduction abelian variety $B$: indeed, $T_{\widehat{\mathbb{Z}}}A$ is a closed subgroup of $T_{\widehat{\mathbb{Z}}}A'$ so that we may view $H$ also as a closed subgroup of $T_{\widehat{\mathbb{Z}}} A'$, and the condition on the dual Hodge-Tate map is evidently equivalent for either. On the other hand, $\widetilde{A}/H$ is the fiber over the identity element under the natural map $\widetilde{A'}/H \rightarrow \widetilde{B}$, so that if the latter is perfectoid so is the former by \cref{lem.product of perfedctoids in proet}. Thus, by replacing $A$ with $(A \times A^\vee)^4$, we may assume $A$ admits a principal polarization $\lambda$ (by Zarhin's trick). Next, by replacing $A$ with $A \times A$ and taking the principal polarization $\begin{bmatrix} 0 & \lambda \\ \lambda & 0 \end{bmatrix}$, we may assume that $H_p$ is isotropic. Finally, we are always free to enlarge $H$, since if $H \leq H'$, then $\widetilde{A}/H$ is a profinite \'{e}tale cover of $\widetilde{A}/H'$, so if the latter is perfectoid then so is the former by almost purity as in \cref{lemma.pro-etale-cover-perf-is-perf}. Thus we may assume that $H_\ell=T_\ell A$ for all $\ell \neq p$ and, using \cref{lemma.maximal-isotropic}, that $H_p$ is maximal isotropic. 

Now, the cover $\widetilde{A}/H$ can be written as the limit of covers as follows: for $H_n\coloneqq H/p^n H \subseteq A[p^n]$, there is a natural map $A/H_n \rightarrow A$ given by quotient by $A[p^n]/H_n$ so that the composition with the quotient map $A \rightarrow A/H_n$ is multiplication by $p^n$ on $A$. This map factors as the iterated composition of the maps $f_n: A_{n+1} \rightarrow A_{n}$ obtained by quotienting by 
\begin{equation}\label{eq.quotient-subgroup} A[p]/H_1 = (A[p^{n+1}]/H_{n+1})[p] \subseteq (A/H_{n+1})[p]. \end{equation}
We have 
\[ \widetilde{A}/H = \left(A \xleftarrow{f_0} A/H_1 \xleftarrow{f_1} A/H_2 \ldots \right).\]
We claim that, for $n \gg 0$, the transition map $f_n$ is a lift of Frobenius mod $p^{1/2}$, and it will follow that $\widetilde{A}/H$ is perfectoid. 

To show this, we will compare with some computations of canonical subgroups in \cite{ScholzeOntorsioninthecohomologyoflocallysymmetricvarieties}. Thus, let $g=\dim_C A$. We may fix a basis $e_1, \ldots, e_g$ of $H_p$ and extend it to a symplectic basis $e_1,\ldots, e_{2g}$ of $T_p A$. Since, by our assumption, the span of $e_1, \ldots, e_g$ in $T_p A \otimes C$ is complementary to $\Lie A(1)$,  there is a unique basis $v_1, \ldots, v_g$ for $\Lie A(1)$ whose vectors in the basis $e_1, \ldots, e_{2g}$ are the columns of a matrix $\begin{bmatrix} M \\ \mathrm{Id}_{g\times g} \end{bmatrix}$ where $M$ is a $g\times g$ matrix with $C$-coefficients, i.e.
\[ v_j = \sum_{i=1}^{g} m_{ij} e_i + e_{g+j}. \]

The isogeny $A/H_n \rightarrow A$ identifies 
\[ T_p (A/H_n) = p^n T_p A + H = \langle e_1, \ldots, e_g, p^n e_{g+1}, \ldots, p^n e_{2g} \rangle.\]
In the basis $e_1, \ldots, e_g, p^ne_{g+1}, \ldots, p^n e_{2g}$, the basis elements $v_i$ for $\Lie A(1)$ have coordinates given by the columns of 
$\begin{bmatrix} M \\ \frac{1}{p^n}\mathrm{Id}_{g\times g} \end{bmatrix}.$ It follows that the associated point in the chart of the flag variety of isotropic subspaces of $C^{2g}$ is given by $\begin{bmatrix} p^nM \\ \mathrm{Id}_{g\times g} \end{bmatrix}$. Thus, as $n \rightarrow \infty$, these approach the point $\begin{bmatrix} 0 \\ \mathrm{Id}_{g\times g} \end{bmatrix}.$ 

It thus follows from the results of \cite[3.2-3.3]{ScholzeOntorsioninthecohomologyoflocallysymmetricvarieties}
that for $n$ sufficiently large, the subgroup $H$ of $T_p (A/H_n)$ is anti-canonical: indeed, the anticanonical locus in the infinite level Siegel modular variety $\mathcal{X}_a\coloneqq \cup_{0 < \epsilon < 1/2} \mathcal{X}(\epsilon)_a$ is open and partially proper, and it contains the closed set $\pi_{\mathrm{HT}}^{-1}(\begin{bmatrix} 0 \\ \mathrm{Id}_{g\times g} \end{bmatrix})$ by Lemma 3.3.20 of loc cit. In particular, the subgroup $(A/H_n) [p] / \left(H_{n+1}/H_n\right) \subseteq (A/H_{n+1})[p]$ is canonical. Since this agrees with the subgroup $A[p]/H_1$ of \cref{eq.quotient-subgroup}, $f_{n+1}$ is a lift of Frobenius. 
\end{proof}

In the proof we used the following simple linear algebra lemma. 

\begin{lemma}\label{lemma.maximal-isotropic}
   Let $A/C$ be a polarized abelian variety and let $H_p \leq T_p A$ be an isotropic subgroup such that $H_p \otimes C \cap \Lie A(1) = \{0\}$. Then there is a maximal isotropic subgroup $H_p'$ containing $H_p$ such that $H_p' \otimes C \cap \Lie A(1) = \{0\}$.
\end{lemma}
\begin{proof}
Let $V=H_p \otimes \mathbb{Q}_p$. It suffices to find a maximal isotropic $V' \supset V$ such that $V'_C \cap \Lie A(1)=\{0\}$, since we may then take $H_p'=V \cap T_p A$. By induction, it suffices to show that if 
\[ \dim_{\mathbb{Q}_p} V =: d < g := \dim A,\]
then there exists an isotropic subspace $W \supsetneq V$ such that $W_C \cap \Lie A(1) = \{0\}$. 

To find such a $W$, we recall that $\Lie A(1)$ is maximal isotropic in $V_C$. We then claim the map $V^\perp_C \rightarrow T_p A \otimes C / \Lie A(1)$ is surjective. We can see this by counting dimensions: 
\[ V^\perp_C \cap \Lie A(1) = V^\perp_C \cap \Lie A(1)^\perp = (V_C \oplus \Lie A(1))^\perp,\]
thus by our assumptions has dimension $2g - (d +g)= (g-d)$. On the other hand, $V^\perp_C$ has dimension $2g-d$, so its image in $T_p A \otimes C / \Lie A(1)$ has dimension $(2g-d) -(g-d)=g$, and thus the map is surjective. Since $d<g$ by assumption, it follows that we may choose a $w \in V^\perp$ such that the image of $w$ in $T_p A \otimes C / \Lie A(1)$ is not in the span of $V_C$. Thus $W=V \oplus \langle w \rangle$ is a larger isotropic subspace such that $W_C \cap \Lie A(1)=\{0\}$, as desired. 
\end{proof}

\begin{remark}\label{remark.abeloid-canonical-subgroup-expected-statement}
Suppose $A/C$ is an abeloid variety of good reduction and dimension $g$, and $H \leq T_p A$ is a subgroup such that $T_p A \otimes C = H \otimes C \oplus \Lie A(1).$ Then for $H_n\coloneqq H/p^n H \subseteq A[p^n]$ and $A_n=A/H_n$, we expect that there exists a pseudouniformizer $\pi$ in $\mathcal{O}_C$ and $N \gg 0$ such that for all $n \geq N$, the natural map $A_{n+1} \rightarrow A_n$ is a lift of Frobenius modulo $\pi$. In the proof of \cref{prop.perfectoid-covers-abelian}, we used the results of \cite{ScholzeOntorsioninthecohomologyoflocallysymmetricvarieties} to establish this for maximal isotropic $H$ when $A$ is a principally polarized abelian variety, then reduced to this case. A direct proof of this expected statement would thus simplify the proof and extend the result to all semi-abeloid varieties. 
\end{remark}

\section{Varieties with globally generated 1-forms}\label{s.varieties-with-globally-generated-1-forms}

By considering the Albanese morphism, we also obtain a perfectoidness result for certain abelian covers of varieties with globally generated 1-forms. We start with the following well-known fact.

\begin{lemma}[{\cite[Lemma 5]{nicholas2000finiteness}}]
\label{lem.pi_1^ab in terms of Alb}
    Let $X/C$ be a smooth proper connected algebraic variety. Fix a point $x \in X(C)$ and its associated Albanese morphism $\iota: X\to \mathrm{Alb}(X)$. Then there is a canonical short exact sequence
    $$0\to \underline{\mathrm{NS}}^\tau(X)^\vee(C)\to \pi_1(X,x)^\mathrm{ab}\to T_{\widehat{\mathbb{Z}}}\mathrm{Alb}(X)\to 0$$
    where $\underline{\mathrm{NS}}^\tau(X)^\vee$ is the Cartier dual of the torsion Néron-Severi group scheme of $X$. In particular, if $\pi_1(X,x)^\mathrm{ab}$ is torsion-free, then $\pi_1(X,x)^\mathrm{ab}\simeq T_{\widehat{\mathbb{Z}}}\mathrm{Alb}(X)$.
\end{lemma}

\begin{lemma}
\label{lem.variety with gg diff unramified}
    Let $X/C$ be a smooth proper connected algebraic variety. Fix a point $x \in X(C)$. Then the Albanese morphism $\iota: X\to \mathrm{Alb}(X)$ is unramified if and only if $\Omega^1_{X/C}$ is generated by global sections.
\end{lemma}
\begin{proof}
    Since $d\iota_y$ differs from $d\iota_x$ by translation, $\iota$ being unramified is equivalent to $T_{X/C,x}\to \mathrm{Lie}\mathrm{Alb}(X)$ being injective by \cite[0B2G]{stacks-project}. It is then equivalent to the dual
    $$d\iota_x^\vee:\mathrm{Lie}\mathrm{Alb}(X)^\vee \simeq \Gamma(\mathrm{Alb}(X),\Omega^1_{\mathrm{Alb}(X)/C})\xrightarrow{\iota^*}\Gamma(X,\Omega^1_{X/C}) \to  T_{X/C,x}^\vee$$
    being surjective. The first isomorphism is a general fact for group varieties, c.f. \cite[III, 5.2]{shafarevich1994basic}. On one hand, $\iota^*$ is injective by \cite[Théorème 4]{serre1958quelques}. On the other hand, the source of $\iota^*$ has dimension $\dim\Pic^0(X)$ which is $\dim_CH^1(X,\mathcal{O}_X)$; and the target of $\iota^*$ also has dimension $\dim_CH^1(X,\mathcal{O}_X)$ by Hodge theory. Hence $\iota^*$ is an isomorphism. Thus $d\iota_x^\vee$ being surjective is equivalent to $\Omega_{X/C}^1$ being generated by global sections.

\end{proof}

\begin{corollary}
\label{corollary.variety with gg diff}
    Let $X/C$ be a smooth proper connected variety with globally generated 1-forms. Fix a point $x \in X(C)$ and denote the corresponding morphism as $\phi:\pi_1(X,x)^\mathrm{ab}\to T_{\widehat{\mathbb{Z}}}\mathrm{Alb}(X)$. Then, for $H$ a closed subgroup of $\pi_{1}(X,x)^\mathrm{ab}$, $X_H^\diamond$ is perfectoid if $\phi(H)_p \otimes C(-1) \cap \Lie \mathrm{Alb}(X)=0$.
\end{corollary}
\begin{proof}
    By \cref{theorem.cover-classification}, the associated cover $\mathrm{Alb}(X)_{\phi(H)}^\diamond\to \mathrm{Alb}(X)^\diamond$ is perfectoid. By \cref{lem.pi_1^ab in terms of Alb}, $X_H$ is a finite \'etale cover of $X_{\phi^{-1}(\phi(H))}$. Hence it suffices to show that $X_{\phi^{-1}(\phi(H))}^\diamond\simeq \iota^*\mathrm{Alb}(X)_H^\diamond$ is perfectoid. By \cref{lem.variety with gg diff unramified}, the Albanese morphism is unramified. Since an unramified morphism is quasi-finite and $X$ is proper, by Zariski's main theorem we find that $X\to \mathrm{Alb}(X)$ is a finite morphism. The statement is local, so it then reduces to show that if $f^\sharp:(R,R^+)\to (S,S^+)$ is a finite morphism of analytic Huber pair with $(R,R^+)$ being perfectoid, then $\mathrm{Spd}(S,S^+)$ can be represented by a perfectoid space. Since $f^\sharp$ is finite, $R^+ \to S^+$ is integral by definition. Hence $(S^+)_\perfd$ is static by \cite[Theorem 10.11]{BhattScholzepPrismaticCohomology}. Thus $\mathrm{Spd}(S,S^+)$ is represented by the perfectoid space $\Spa((S^+)_\perfd[1/p],(S^+)_\perfd)$ as desired.
\end{proof}
\begin{remark}\label{remark.unramified-alb-better-version} 
   Note that, because of our use of the perfectoidization of a finite morphism in the proof \cref{corollary.variety with gg diff}, in general we only obtain that $X_H^\diamond$ is perfectoid, not that $X_H$ is perfectoid. However, in certain special cases, one can show that in fact $X_H$ is perfectoid. For example, when $X$ is a smooth proper curve of genus greater or equal to 1, the Albanese map is a Zariski closed immersion and by using \cref{lem.zariski closed of perfectoid in proet}, one obtains that $X_H$ is perfectoid.  Generalizing the case of curves, if the image of Albanese is smooth and $X$ has globally generated 1-forms, then by miracle flatness, the Albanese morphism is the composition of a finite \'etale morphism and a closed immersion, and thus one can upgrade to $X_H$ being perfectoid by the same argument. 
\end{remark}

\begin{remark}\label{remark.distinguishing-perfectoid-conditions}
Using a similar argument to the proof \cref{corollary.variety with gg diff}, one can produce profinite \'{e}tale covers $X_\infty$ of smooth rigid analytic varieties $X$ such that $X_\infty^\diamond$ is perfectoid but $X_\infty$ is not. For example, for $0 < \epsilon \leq 1$, let $\mathbb{D}_{\epsilon}$ be the closed disk of radius $\epsilon$ around $0$, and let $g: \mathbb{D} \rightarrow \mathbb{D}_{\epsilon/2}$ be defined by $g(t)=t^2$. Writing $\mathbb{D}_{\epsilon/2,\infty}$ for the usual tower obtained by extracting $p^n$th power roots of $(1+t)$, we have $\mathbb{D}_{\epsilon/2,\infty}$ is perfectoid. Since $g$ is finite, it follows that $g^* \mathbb{D}_{\epsilon/2,\infty}^\diamond$ is perfectoid. However, we claim $\mathbb{D}\times_{\mathbb{D}_{\epsilon/2}}\mathbb{D}_{\epsilon/2,\infty}$ is not perfectoid. Note that this is predicted by \cite[Conjecture 3.3.5]{JuangeometricSen}: since $dg|_0=0$, the functoriality of the geometric Sen morphism implies that the geometric Sen morphism for $\mathbb{D}\times_{\mathbb{D}_{\epsilon/2}}\mathbb{D}_{\epsilon/2,\infty}$ vanishes on the fiber at $0 \in \mathbb{D}_\epsilon$. To verify that it is indeed not perfectoid, we claim that the derivative with respect to $t$ at $0$ defines a non-zero continuous derivation on 
on the completed direct limit of functions on the tower
$\mathbb{D}\times_{\mathbb{D}_{\epsilon/2}}\mathbb{D}_{\epsilon/2,\infty}$. This contradicts the perfectoidness of $\mathbb{D}\times_{\mathbb{D}_{\epsilon/2}}\mathbb{D}_{\epsilon/2,\infty}$ since it applies for any $\epsilon$ and after restricting first to any finite level, thus no open containing a pre-image of $0$ can be affinoid perfectoid. That one obtains a continuous derivation follows from the computation $\partial_t|_{t=0} (1+t^2)^{1/p^n}=0.$ This justification is closely related to our proof of the non-perfectoidness direction of \cref{theorem.cover-classification}, where we also exhibit non-trivial derivations. 
\end{remark}

Similarly, one obtains a computation of the geometric Sen morphism for abelian $p$-adic Lie torsors over such varieties by considering the morphism to Albanese as the composition of $d\rho$ with the constant inclusion $\Lie\mathrm{Alb}(X) \rightarrow T_p \mathrm{Alb}(X) \otimes C(-1)$ and the non-constant  $d\iota$ for $\iota: X \rightarrow \mathrm{Alb}(X)$ the Albanese morphism associated to the point $x$ (which is independent of $x$). The assignment sending $x \in X$ to the subspace spanned by the image of $d\iota_x$ is the canonical map for $X$. Thus, the natural conjecture is: 

\begin{conjecture}\label{conj.varieties}
   Let $X/C$ be a smooth proper connected varieties of pure dimension $d$ with globally generated 1-forms. Fix a point $x \in X(C)$. Then, for $H=\prod H_\ell$ a closed subgroup of $\pi_{1}(x,X)^\mathrm{ab}=T_{\widehat{\mathbb{Z}}} \mathrm{Alb}(X)$, $X_H|_U$ is perfectoid for $U \subseteq X$ the open subvariety obtained as the preimage under the canonical map $X \rightarrow \mathrm{Gr}_d(\Lie \mathrm{Alb}(X))$ of the open set of $d$-dimensional subspaces in $\Lie \mathrm{Alb}(X)$ that intersect $H_p \otimes C(-1)$ trivially in $T_p \mathrm{Alb}(X) \otimes C(-1)$. Moreover, no open perfectoid subdiamond of $X_H^\diamond$ has non-empty intersection with the restriction of $X_H^\diamond$ to the complementary closed subvariety in $X$.
\end{conjecture}

\begin{remark}It has been relayed to us through Ben Heuer and Peter Wear that Piotr Achinger has asked for a characterization of the perfectoid profinite \'{e}tale covers of curves and abelian varieties. Our main result \cref{theorem.cover-classification} gives a complete answer for abelian varieties, while the discussion above and a comparison of the expected \cref{conj.varieties} to the partial result given by \cref{corollary.variety with gg diff} shows that even the case of abelian covers of curves is in some ways more subtle. 
\end{remark}

\appendix
\section{Pro-finite Étale Fundamental Groups of Abeloid Varieties}\label{appedix.abeloid-fundamental-group}

In order to study $\mathbb{Z}_p$-local systems on abeloid varieties, we must understand their profinite \'etale fundamental groups.  In the algebraic setting, ~\cite[\textsection IV]{MumfordAbVar} shows that the profinite \'etale fundamental group $\pi_1(A,e)$ of an abelian variety $A$ over an algebraically closed field of characteristic zero is the adelic Tate module $T_{\widehat{\mathbb{Z}}}A$.  However, the arguments given rely on the existence of an ample line bundle, which is special to the algebraic setting.  In this appendix, we provide a proof of the analogous result for abeloid varieties over $C$ (our fixed algebraically closed non-archimedan extension of $\mathbb{Q}_p$). This result is \cref{prop.profinite-etale-fundamental-group}. 

\begin{remark} Abeloid varieties can, in general, have irreducible de Jong \'etale covers of infinite degree: for example, the Tate uniformization of an elliptic curve with multiplicative reduction. As illustrated by this example, such covers arise already for the analytifications of abelian varieties. These covers are not seen by the profinite \'{e}tale fundamental group. 
\end{remark}

Lang and Serre showed that a finite \'etale cover of an abelian variety over an algebraically closed field is again an abelian variety; we begin by providing the analogous result for abeloid varieties.
\begin{lemma}\label{lemma.abeloid-variety-finite-etale-cover-group-structure}
    Let $A$ be an abeloid variety over a non-archimedean extension $C$.  Then given a finite \'etale cover $\pi:X\rightarrow A$, where $X$ is irreducible, $X$ can be equipped with the structure of an abeloid variety so that $\pi$ is a morphism of rigid analytic groups.
\end{lemma}
\begin{proof}
    We sketch the argument given in ~\cite[\textsection 18]{MumfordAbVar}.  Let $\Gamma_A$ be the graph (inside $A\times A\times A$) of the multiplication map $m:A\times A\rightarrow A$, and let $\Gamma_X$ be its pre-image (in $X\times X\times X$) under $\pi\times\pi\times \pi$.  To see that an irreducible component of $\Gamma_X$ is the graph of a morphism, we want to know that the restriction of the projection map $p_{X,12}:\Gamma_X\rightarrow X\times X$ is an isomorphism.

    We have a commutative diagram
    \[
    \begin{tikzcd}
        \Gamma_X \ar[r]\ar[d] & \Gamma_A \ar[d]   \\
        X\times X\ar[r] & A\times A
    \end{tikzcd}
    \]
    where the horizontal arrows are \'etale and the right vertical arrow is an isomorphism, so the left vertical arrow is \'etale, as well.

    Choose a point $x_0\in X(C)$ such that $\pi(x_0)$ is the identity of $A$, and let $\Gamma_X^0$ be the component of $\Gamma_X$ containing $(x_0,x_0,x_0)$.  Then we have a morphism $\sigma:X\rightarrow \Gamma_X^0$ given by $\sigma(x)=(x,x_0,x)$; this is a section for the composite map
    \[  q:\Gamma_X^0\rightarrow X\times X\xrightarrow{p_1} X  \]
    Now $q$ is smooth (since it is the composite of an \'etale morphism and a projection map).

    Moreover, $X$ is proper (since it is finite over a proper variety), so $q$ is also proper (since it is the composition of the closed immersion $\Gamma_X^0\hookrightarrow X\times X\times X$ with the proper projection map $X\times X\times X\rightarrow X$).  We claim the existence of the section $\sigma$ implies that the fibers of $q$ are irreducible.  Then the preimage of $X\times\{x_0\}$ in $\Gamma_X^0$ is $\sigma(X)$, so the fibers of $\Gamma_X^0\rightarrow X\times X$ are irreducible, hence single points.

    Since $q:\Gamma_X^0\rightarrow X$ is proper, it factors as
    \[  \Gamma_X^0\rightarrow X'\rightarrow X   \]
    where $\mathscr{O}_{X'}=q_\ast \mathscr{O}_{\Gamma_X^0}$ and $X'\rightarrow X$ is finite.  It also admits a section, and the irreducible component of $X'$ which contains the image of $\Gamma_X^0$ is therefore isomorphic to $X$.  Then we use the theorem on formal functions (and the irreducibility of $X$) to see that the fibers of $q$ are irreducible.

    Now we know that $\Gamma_X^0$ is the graph of a morphism $m_X:X\times X\rightarrow X$.  It is straightforward to check that $m|_{X\times \{x_0\}}$ and $m|_{\{x_0\}\times X}$ are the identity morphisms.

    Finally, we need to check the existence of the inverse morphism.  This follows as in ~\cite[Appendix to \textsection 4]{MumfordAbVar}, with the minor detail that the use of the rigidity lemma must be replaced by a rigid analytic version~\cite[Proposition 7.1.2]{lutkebohmert2016}.
\end{proof}

\begin{remark}
    We do not know whether there is a version of this theorem if $\pi:X\rightarrow A$ is only assumed to be, e.g., a de Jong \'{e}tale covering space.  It would be nicer to know that, for example, the uniformization map $\mathbb{G}_m^{\mathrm{ad}}\rightarrow E$ of a Tate curve automatically equips $\mathbb{G}^{\mathrm{ad}}_m$ with the structure of a rigid analytic group variety.
\end{remark}

We will also need to know that the multiplication-by-$N$ map on an abeloid variety is an isogeny (that is, finite and surjective).  In the algebraic setting, Mumford proves this using the existence of an ample line bundle (which does not exist on a non-algebraic abeloid variety).  Instead, we exploit the \'etaleness of the multiplication map (note that any finite \'{e}tale map with irreducible target is, in particular, surjective).  However, this requires $N$ to be prime to the characteristic of the ground field, so we do use here that $C$ is assumed to be an extension of $\mathbb{Q}_p$.

\begin{lemma}\label{lemma.mult-by-N-isogeny}
    For $A/C$ an abeloid variety and $N \geq 1$, the multiplication by $N$ map $[N]_A$ is finite \'{e}tale.
\end{lemma}
\begin{proof}
    We first observe that $[N]_A$ is proper (and hence quasi-compact).  Indeed, $[N]_A$ factors as
\[	A\xrightarrow{\Gamma_{[N]_A}}A\times_{\Spa K}A\xrightarrow{\mathrm{pr}_2}A	\]
where the first map is the graph morphism and the second is projection.  Since $A\rightarrow\Spa C$ is separated, the graph morphism is a closed immersion (hence proper), and $\mathrm{pr}_2$ is the base-change of the proper morphism $A\rightarrow \Spa C$.  Hence the composition is proper.
Now, $[N]_A$ is \'etale, since it induces an isomorphism on the tangent space of $A$.  But \'etale morphisms of rigid spaces are locally (on the source) quasi-finite; quasi-compactness implies that $\ker [N]_A$ is finite.  Hence $[N]_A$ is finite \'{e}tale.
\end{proof}

\begin{proposition}\label{prop.profinite-etale-fundamental-group}
    The profinite \'etale fundamental group $\pi_1(A, e_A)$ of an abeloid variety $A/C$ is $T_{\widehat{\mathbb{Z}}}A$. 
\end{proposition}
\begin{proof}
It suffices to show that, for any finite \'{e}tale cover $f: A' \rightarrow A$, we can find a finite \'{e}tale $A \xrightarrow{g} A'$ such that $f \circ g$ is $[N]_A$ for some $N$.  Given such a finite \'etale cover $f:A'\rightarrow A$ of degree $N$, we have seen in \cref{lemma.abeloid-variety-finite-etale-cover-group-structure} that $A'$ must be an abeloid variety and that $f$ is a morphism.  The multiplication map $[N]_{A'}:A'\rightarrow A'$ annihilates the kernel of $A'\rightarrow A$, so $[N]_{A'}$ factors as $A'\xrightarrow f A\xrightarrow g A'$ for some morphism $g$, which is finite \'{e}tale because $[N]_{A'}$ is finite \'{e}tale by \cref{lemma.mult-by-N-isogeny}.  Since 
\[ f \circ g \circ f= f \circ [N]_{A'}= [N]_{A} \circ f, \]
we conclude $f \circ g= [N]_A$, as desired. 
\end{proof}

\def\Luoma#1{\uppercase\expandafter{\romannumeral#1}}
\def\luoma#1{\romannumeral#1}
\section{Perfectoid Covers of Abelian Varieties via Sen Theory and Purity, by Tongmu He}\label{appendix.he}

In this appendix, we give a different proof of the special case of \cref{prop.perfectoid-covers-abelian} when the semi-abelian rigid analytic variety $E$ over $C$ is defined over a discrete valuation subfield. This proof requires neither the reduction to the maximal isotropic subgroup nor Scholze's results on canonical subgroups. Instead, we use the stalkwise perfectoidness via Sen theory of \cite{he2024perfd} and the purity for perfectoidness of \cite{he2024purity}. To invoke the stalkwise perfectoidness results, one needs only the computation of the Sen morphism given in \cref{prop.geometric-sen-semi-abeloid}. To invoke the purity of perfectoidness, one needs to know furthermore that the tower of abelian varieties in question admits a sufficiently nice integral model. We establish this below in the good reduction case, then reduce to the good reduction case as in the earlier proof of \cref{prop.perfectoid-covers-abelian}. 

Let $K$ be a complete discrete valuation field extension of $\mathbb{Q}_p$ with perfect residue field, and let $\overline{K}$ be an algebraic closure of $K$.

\begin{lemma}\label{lem:ab-var-int-model}
	Let $\mathcal{A}$ be an abelian scheme over $\mathcal{O}_{\overline{K}}$, $A=\mathcal{A}_\eta$ the generic fibre of $\mathcal{A}$, $H\subseteq A[p^n]$ a closed subgroup scheme over $\overline{K}$ (where $n\in\mathbb{N}$).
	\begin{enumerate}
		\renewcommand{\labelenumi}{{\rm(\theenumi)}}
		\item The scheme theoretic closure $\mathcal{H}$ of $H$ in $\mathcal{A}[p^n]$ is a closed subgroup scheme of $\mathcal{A}[p^n]$ finite flat over $\mathcal{O}_{\overline{K}}$ with generic fibre $\mathcal{H}_\eta=H$.
		\item The fppf quotient $\mathcal{A}/\mathcal{H}$ is represented by an abelian scheme over $\mathcal{O}_{\overline{K}}$.
	\end{enumerate}
\end{lemma}
\begin{proof}
	After enlarging $K$, we may assume that $\mathcal{A}$ and $H$ are defined over $\mathcal{O}_K$. As $H\to \mathcal{A}[p^n]$ is a morphism of group schemes over $\mathcal{O}_K$, it is clear that the scheme theoretic closure $\mathcal{H}$ of $H$ in $\mathcal{A}[p^n]$ is also a group scheme over $\mathcal{O}_K$ (\cite[\Luoma{8}.7.1]{sga3-2}). Moreover, since $H$ and $\mathcal{A}[p^n]$ are finite flat over $K$ and $\mathcal{O}_K$ respectively, we see that $\mathcal{H}$ is finite flat over $\mathcal{O}_K$ with $\mathcal{H}_\eta=H$. 
	
	Then, by a representability theorem of Grothendieck (\cite[\Luoma{5}.4.1]{sga3-1}, see also \cite[3.5]{tate1997flat}), the fppf quotient $\mathcal{A}/\mathcal{H}$ is represented by a group scheme over $\mathcal{O}_K$ and moreover $\mathcal{A}\to \mathcal{A}/\mathcal{H}$ is finite and faithfully flat. Thus, $\mathcal{A}/\mathcal{H}$ is proper smooth with geometrically connected fibres as $\mathcal{A}$ is so by faithfully flat descent and the Noetherianity of $\mathcal{O}_K$ (\cite[\href{https://stacks.math.columbia.edu/tag/05B5}{05B5}, \href{https://stacks.math.columbia.edu/tag/0CYQ}{0CYQ}]{stacks-project}). In other words, $\mathcal{A}/\mathcal{H}$ is an abelian scheme over $\mathcal{O}_K$. The conclusion follows after taking the flat base change of $\mathcal{H}$ along $\mathcal{O}_K\to \mathcal{O}_{\overline{K}}$.
\end{proof}

\begin{proposition}\label{prop:ab-var-int-model}
	Let $\mathcal{A}$ be an abelian scheme over $\mathcal{O}_{\overline{K}}$, $A=\mathcal{A}_\eta$ the generic fibre of $\mathcal{A}$, $\widetilde{A}=\lim_{[p]}A$ the pro-finite \'etale $A$-scheme with Galois group $T_pA$, $\widetilde{A}/H$ the pro-finite \'etale $A$-scheme corresponding to a closed subgroup $H\subseteq T_pA$. Then, there is a canonical $\mathcal{O}_{\overline{K}}$-scheme $\widetilde{\mathcal{A}}/\mathcal{H}$ such that 
	\begin{enumerate}
		\renewcommand{\labelenumi}{{\rm(\theenumi)}}
		\item its generic fibre $(\widetilde{\mathcal{A}}/\mathcal{H})_\eta=\widetilde{A}/H$, and that
		\item it is a cofiltered limit of abelian schemes over $\mathcal{O}_{\overline{K}}$ whose transition morphisms are finite, faithfully flat, and finite \'etale over $\overline{K}$.
	\end{enumerate}
\end{proposition}
\begin{proof}
	Let $H_n$ be the image of $H\to A[p^n](\overline{K})$. Since $T_pA\cong \mathbb{Z}_p^{2g}$ with $A[p^n](\overline{K})\cong T_pA/p^nT_pA\cong (\mathbb{Z}/p^n\mathbb{Z})^{2g}$ (where $g=\dim A$, \cite[\href{https://stacks.math.columbia.edu/tag/03RP}{03RP}]{stacks-project}), we have $H=\lim_{n\in\mathbb{N}}H/H\cap p^nT_pA=\lim_{[p]}H_n$. Recall that the category of finite group schemes over $\overline{K}$ (which are automatically \'etale, see \cite[\href{https://stacks.math.columbia.edu/tag/047N}{047N}]{stacks-project}) is equivalent to the category of finite groups by taking the group of $\overline{K}$-points. Thus, we may regard $H_n$ as a closed subgroup scheme of $A[p^n]$ and we have $\widetilde{A}/H=\lim_{[p]} A/H_n$.
	
	Let $\mathcal{H}_n$ be the scheme theoretic closure of $H_n$ in $\mathcal{A}[p^n]$ and $\widetilde{\mathcal{A}}=\lim_{[p]}\mathcal{A}$. Then, the fppf quotient $\mathcal{A}/\mathcal{H}_n$ is an abelian scheme over $\mathcal{O}_{\overline{K}}$ with generic fibre $(\mathcal{A}/\mathcal{H}_n)_\eta=A/H_n$ by Lemma \ref{lem:ab-var-int-model}. Moreover, $(\mathcal{A}/\mathcal{H}_n)_{n\in\mathbb{N}}$ forms a directed inverse system of abelian schemes over $\mathcal{O}_{\overline{K}}$ with transition morphisms given by $[p]$. Since $[p]:\mathcal{A}\to \mathcal{A}$ and $\mathcal{A}\to \mathcal{A}/\mathcal{H}_n$ are finite, faithfully flat, and finite \'etale after inverting $p$ (see \cite[\href{https://stacks.math.columbia.edu/tag/0BFG}{0BFG}, \href{https://stacks.math.columbia.edu/tag/0BFH}{0BFH}, \href{https://stacks.math.columbia.edu/tag/039C}{039C}]{stacks-project} and the proof of Lemma \ref{lem:ab-var-int-model}), so is $[p]:\mathcal{A}/\mathcal{H}_{n+1}\to \mathcal{A}/\mathcal{H}_n$ for any $n\in\mathbb{N}$. Therefore, the conclusion follows by taking $\widetilde{\mathcal{A}}/\mathcal{H}=\lim_{[p]}\mathcal{A}/\mathcal{H}_n$.
\end{proof}

\begin{remark}\label{rem:ab-var-int-model-K}
	Assume that $\mathcal{A}$ is defined over $\mathcal{O}_K$. Let $K_\infty$ be the Galois extension of $K$ such that $\mathrm{Gal}(\overline{K}/K_\infty)$ is the kernel of the continuous homomorphism $\mathrm{Gal}(\overline{K}/K)\to \mathrm{Aut}_{\mathbb{Z}_p}(T_pA)\cong \mathrm{GL}_{2g}(\mathbb{Z}_p)$. Thus, $\mathrm{Gal}(K_\infty/K)\subseteq \mathrm{GL}_{2g}(\mathbb{Z}_p)$ is a $p$-adic analytic group (\cite[3.11]{he2022sen}) and $A[p^n](\overline{K})=A[p^n](K_\infty)$. Therefore, $H_n\subseteq A[p^n]$ is defined over $K_\infty$ and thus so are $A/H_n$, $\widetilde{A}/H$, $\mathcal{H}_n$, $\mathcal{A}/\mathcal{H}_n$ and $\widetilde{\mathcal{A}}/\mathcal{H}$.
\end{remark}

\begin{remark}\label{rem:ab-var-int-model}
	Note that each fibre of an abelian scheme is irreducible. The transition morphism $[p]:\mathcal{A}/\mathcal{H}_{n+1}\to \mathcal{A}/\mathcal{H}_n$ preserves the unique generic point of the generic (resp. special) fibre as it is faithfully flat (\cite[3.5.(1)]{he2024purity}). In particular, the generic (resp. special) fibre of the $\mathcal{O}_{\overline{K}}$-scheme $\widetilde{\mathcal{A}}/\mathcal{H}$ is also irreducible whose generic point $y$ (resp. $x$) lies over that of each $\mathcal{A}/\mathcal{H}_n$ (\cite[3.7]{he2024purity}). We remark that the local ring $\mathcal{O}_{\widetilde{\mathcal{A}}/\mathcal{H},x}$ is a valuation ring of height $1$ extension of $\mathcal{O}_{\overline{K}}$ and its fraction field $\mathcal{O}_{\widetilde{\mathcal{A}}/\mathcal{H},x}[1/p]$ is the residue field $\kappa(y)$ of $\widetilde{A}/H$ at $y$ (\cite[8.4.(3)]{he2024purity}).
\end{remark}

\begin{definition}\label{defn:perfd-scheme}
	Let $X$ be an $\mathcal{O}_{\overline{K}}$-scheme. We say that $X$ is \emph{pre-perfectoid} if every affine open subset of $X$ is the spectrum of a pre-perfectoid $\mathcal{O}_{\overline{K}}$-algebra in the sense of \cite[2.3]{he2024purity}.
\end{definition}

It is easy to see that $X$ is pre-perfectoid if and only if it admits an affine open covering by spectra of pre-perfectoid $\mathcal{O}_{\overline{K}}$-algebras.

\begin{corollary}\label{cor:perfd-purity}
	With the notation in {\rm Proposition \ref{prop:ab-var-int-model}}, let $x$ be the unique generic point of the special fibre of the $\mathcal{O}_{\overline{K}}$-scheme $\widetilde{\mathcal{A}}/\mathcal{H}$ {\rm (Remark \ref{rem:ab-var-int-model})}. If $\mathcal{O}_{\widetilde{\mathcal{A}}/\mathcal{H},x}$ is pre-perfectoid, then $\widetilde{\mathcal{A}}/\mathcal{H}$ is pre-perfectoid.
\end{corollary}
\begin{proof}
	We take an affine open covering of $\mathcal{A}$. Then, it induces an affine open covering $\bigcup_{i\in I}\mathrm{Spec}(R_i)$ of $\widetilde{\mathcal{A}}/\mathcal{H}$ by the $0$-th projection $\widetilde{\mathcal{A}}/\mathcal{H}\to \mathcal{A}$. Applying the purity for perfectoidness to each ind-smooth algebra $R_i$ \cite[8.6]{he2024purity}, we see that $R_i$ is pre-perfectoid and thus $\widetilde{\mathcal{A}}/\mathcal{H}$ is pre-perfectoid.
\end{proof}

\begin{lemma}\label{lem:perfd-analytification}
	Let $X$ be a flat $\mathcal{O}_{\overline{K}}$-scheme, $\widehat{X}$ its associated $p$-adic formal scheme over $\mathrm{Spf}(\mathcal{O}_{\widehat{\overline{K}}})$ {\rm\cite[2.5.2]{abbes2010rigid}}, $\widehat{X}^{\mathrm{ad}}$ the associated adic space over $\mathrm{Spa}(\mathcal{O}_{\widehat{\overline{K}}},\mathcal{O}_{\widehat{\overline{K}}})$ {\rm \cite[2.2.1]{scholze2012moduli}}, \[\widehat{X}^{\mathrm{ad}}_\eta=\widehat{X}^{\mathrm{ad}}\times_{\mathrm{Spa}(\mathcal{O}_{\widehat{\overline{K}}},\mathcal{O}_{\widehat{\overline{K}}})}\mathrm{Spa}(\widehat{\overline{K}},\mathcal{O}_{\widehat{\overline{K}}}).\]
    If $X$ is pre-perfectoid, then $\widehat{X}^{\mathrm{ad}}_\eta$ is perfectoid in the sense of \cite[2.3.3]{scholze2012moduli}.
\end{lemma}
\begin{proof}
	Let $\bigcup_{i\in I}\mathrm{Spec}(R_i)$ be an affine open covering of $X$ and let $R_i^+$ be the integral closure of the $p$-adic completion $\widehat{R_i}$ in $\widehat{R_i}[1/p]$. Then, $(\widehat{R_i}[1/p],R_i^+)$ is an affinoid $(\widehat{\overline{K}},\mathcal{O}_{\widehat{\overline{K}}})$-algebra (\cite[5.40]{he2024coh}). Moreover, it is a perfectoid affinoid $(\widehat{\overline{K}},\mathcal{O}_{\widehat{\overline{K}}})$-algebra as $R_i$ is a pre-perfectoid $\mathcal{O}_{\overline{K}}$-algebra (\cite[5.39]{he2024coh}).  By construction, $\bigcup_{i\in I}\mathrm{Spa}(\widehat{R_i}[1/p],R_i^+)$ forms an open covering of $\widehat{X}^{\mathrm{ad}}_\eta$, which shows that the latter is perfectoid.
\end{proof}

\begin{remark}\label{rem:perfd-analytification}
	We keep the notation in Lemma \ref{lem:perfd-analytification}.
	\begin{enumerate}
		\renewcommand{\labelenumi}{{\rm(\theenumi)}}
		\item Assume that $X$ is separated of finite type over $\mathcal{O}_{\overline{K}}$. Then, $\widehat{X}^{\mathrm{ad}}_\eta$ is also the adic space associated to the rigid analytic generic fibre $\widehat{X}^{\mathrm{rig}}$ of the formal scheme $\widehat{X}$ (\cite[4.3]{huber1994general}, \cite[4.1]{bosch1993rigid1}). Let $X_\eta^{\mathrm{ad}}=X_\eta\times_{\mathrm{Spec}(\overline{K})}\mathrm{Spa}(\widehat{\overline{K}},\mathcal{O}_{\widehat{\overline{K}}})$ be the analytification of the generic fibre $X_\eta$ of $X$ as an adic space (\cite[3.8]{huber1994general}). Then, there is a canonical open immersion of adic spaces $\widehat{X}^{\mathrm{ad}}_\eta\to X_\eta^{\mathrm{ad}}$, which is an isomorphism if $X$ is proper over $\mathcal{O}_{\overline{K}}$ (\cite[4.6.(\luoma{1})]{huber1994general}, \cite[\Luoma{2}.9.1.5]{fujiwarakato2018rigid}).\label{item:rem:perfd-analytification-1}
		\item Assume that $X$ is the limit of a directed inverse system $(X_\lambda)_{\lambda\in\Lambda}$ of flat $\mathcal{O}_{\overline{K}}$-schemes with affine transition morphisms. Then, $\widehat{X}^{\mathrm{ad}}_\eta\sim \lim_{\lambda\in\Lambda}\widehat{X_\lambda}^{\mathrm{ad}}_\eta$ in the sense of \cite[2.4.1]{scholze2012moduli}. Indeed, it is easy to reduce to the affine case and thus to \cite[5.29]{he2024coh} and \cite[2.4.2]{scholze2012moduli}. \label{item:rem:perfd-analytification-2}
		\item Following (\ref{item:rem:perfd-analytification-2}), if moreover each $X_\lambda$ is proper over $\mathcal{O}_{\overline{K}}$ and if the transition morphisms of the generic fibres $(X_{\lambda,\eta})_{\lambda\in\Lambda}$ are finite \'etale surjective, then for any $\lambda_0\in\Lambda$, $(X_{\lambda,\eta}^{\mathrm{ad}})_{\lambda\in\Lambda_{\geq \lambda_0}}$ forms an object of the pro-\'etale site of $X_{\lambda_0,\eta}^{\mathrm{ad}}$ (\cite[3.9]{scholze2013adic}). Combining Lemma \ref{lem:perfd-analytification}, (\ref{item:rem:perfd-analytification-1}) and (\ref{item:rem:perfd-analytification-2}), we see that $(X_{\lambda,\eta}^{\mathrm{ad}})_{\lambda\in\Lambda_{\geq \lambda_0}}$ is perfectoid in the sense of \cite[4.3]{scholze2013adic} if $X$ is pre-perfectoid. \label{item:rem:perfd-analytification-3}
	\end{enumerate}
	
\end{remark}

\begin{proposition}\label{prop:perfd-cov-ab-var}
	Let $E$ be a semi-abelian scheme over $\overline{K}$, $E^{\mathrm{ad}}$ its analytification as an adic space, $\widetilde{E}^{\mathrm{ad}}=\lim_{[n]}E^{\mathrm{ad}}\in E^{\mathrm{ad}}_{\mathrm{pro\acute{e}t}}$ the pro-finite \'etale cover of $E^{\mathrm{ad}}$ with Galois group $T_{\widehat{\mathbb{Z}}}E^{\mathrm{ad}}$ {\rm(\cref{def.prof-etale-Etilde-cover})}, $E^{\mathrm{ad}}_H\in E^{\mathrm{ad}}_{\mathrm{pro\acute{e}t}}$ the pro-finite \'etale cover of $E^{\mathrm{ad}}$ corresponding to a closed subgroup $H=\prod_{\ell\ \mathrm{ prime}} H_{\ell}\subseteq T_{\widehat{\mathbb{Z}}}E^{\mathrm{ad}}=\prod_{\ell\ \mathrm{ prime}} T_{\ell}E^{\mathrm{ad}}$. Assume that $f^{-1}(H_p\otimes \widehat{\overline{K}}(-1))=0$, where $f:\mathrm{Lie}E^{\mathrm{ad}}\to T_pE^{\mathrm{ad}}\otimes_{\mathbb{Z}_p}\widehat{\overline{K}}(-1)$ is the dual Hodge-Tate map (see {\rm \cref{theorem.fargues-rig-an-dual-HT}} and {\rm Construction \ref{cons.p-divisible group attached to abeloid}}). Then, $E^{\mathrm{ad}}_H$ is perfectoid.
\end{proposition}
\begin{proof}
	Firstly, we make some simplifications as in the proof of \cref{prop.perfectoid-covers-abelian}. By almost purity (\cref{lemma.pro-etale-cover-perf-is-perf}), we may assume that $H_{\ell}=T_{\ell}E^{\mathrm{ad}}$ for any prime number $\ell\neq p$ so that $H=H_p\subseteq T_pE^{\mathrm{ad}}$. As $E$ is an extension of an abelian scheme $A$ by a torus $T$ over $\overline{K}$, its analytification $E^{\mathrm{ad}}$ is an extension of an abelian rigid analytic variety $A^{\mathrm{ad}}$ by a rigid analytic torus $T^{\mathrm{ad}}$ over $\mathrm{Spa}(\widehat{\overline{K}},\mathcal{O}_{\widehat{\overline{K}}})$ (\cite[\Luoma{2}.9.1.10]{fujiwarakato2018rigid}). Thus, we may assume that $E=A$ by \cref{lemma.semi-abelioid-to-abeloid}. Using Raynaud's uniformization and its algebraization \cite[6.6.1]{lutkebohmert2016}, there exists a semi-abelian scheme $\mathcal{E}$ extension of an abelian scheme $\mathcal{A}$ by a torus $\mathcal{T}$ over $\mathcal{O}_{\overline{K}}$ such that $E$ is isomorphic to the quotient of the generic fibre $\mathcal{E}_\eta$ by a discrete free abelian subgroup $M$. Thus, $E^{\mathrm{ad}}=\mathcal{E}_\eta^{\mathrm{ad}}/M$ so that we may assume that $E=\mathcal{E}_\eta$ by \cref{lemma.semi-abeloid-to-quotient}. Finally, using again \cref{lemma.semi-abelioid-to-abeloid}, we may assume that $E=A=\mathcal{A}_\eta$.
	
	We put $\widetilde{A}=\lim_{[p]}A$ the pro-finite \'etale $A$-scheme with Galois group $T_pA$ and $\widetilde{A}/H$ the pro-finite \'etale $A$-scheme corresponding to $H\subseteq T_pA$. Let $\widetilde{\mathcal{A}}/\mathcal{H}$ be the canonical $\mathcal{O}_{\overline{K}}$-scheme constructed in Proposition \ref{prop:ab-var-int-model}. It remains to show that $\widetilde{\mathcal{A}}/\mathcal{H}$ is pre-perfectoid by Remark \ref{rem:perfd-analytification}.(\ref{item:rem:perfd-analytification-3}).
	
	Let $y$ (resp. $x$) be the unique generic point of the generic (resp. special) fibre of the $\mathcal{O}_{\overline{K}}$-scheme $\widetilde{\mathcal{A}}/\mathcal{H}$ {\rm (Remark \ref{rem:ab-var-int-model})}. Let $\overline{F}$ be an algebraic closure of the residue field $F$ at $y\in \widetilde{A}/H$ and let $\mathcal{O}_{\overline{F}}$ be a valuation ring of $\overline{F}$ extension of $\mathcal{O}_F=\mathcal{O}_{\widetilde{\mathcal{A}}/\mathcal{H},x}$. Then, $\overline{y}=\mathrm{Spec}(\overline{F})\to \widetilde{A}/H$ is a geometric valuative point over $\mathcal{O}_{\overline{K}}$ in the sense of \cite[7.1]{he2024perfd}. Consider the following canonical diagram of $\widehat{\overline{F}}$-modules:
	\begin{align}
		\xymatrix{
			\mathrm{Hom}_{\mathcal{O}_A}(\Omega^1_{A/\overline{K}},\widehat{\overline{F}})\ar[rr]^-{\varphi_{\mathrm{Sen}}^{\mathrm{geo}}|_{T_pA/H,\overline{y}}}\ar[d]^-{\wr}&& T_pA/H\otimes_{\mathbb{Z}_p}\widehat{\overline{F}}(-1)\\
			\mathrm{Lie}A\otimes_{\overline{K}}	\widehat{\overline{F}}\ar[rr]^-{ f\otimes_{\widehat{\overline{K}}}\mathrm{id}_{\widehat{\overline{F}}}}&&T_pA\otimes_{\mathbb{Z}_p}\widehat{\overline{F}}(-1)\ar[u]
		}
	\end{align}
	where the left vertical arrow is the canonical isomorphism (\cite[\href{https://stacks.math.columbia.edu/tag/047I}{047I}]{stacks-project}), the right vertical arrow is the canonical surjection, the bottom arrow is the base change of the dual Hodge-Tate map (where we identify $\mathrm{Lie}A^{\mathrm{ad}}=\mathrm{Lie}A\otimes_{\overline{K}}\widehat{\overline{K}}$ and $T_pA^{\mathrm{ad}}=T_pA$), and the top arrow is the geometric Sen morphism of the $p$-adic analytic Galois extension $\widetilde{A}/H$ of $A$ at $\overline{y}$ constructed in \cite[8.5]{he2024perfd} (see also \cite{he2022sen}) (note that the pro-finite \'etale cover $\widetilde{A}/H\to A$ is defined over a $p$-adic analytic Galois extension of $K$, see Remark \ref{rem:ab-var-int-model-K}). This diagram is commutative by \cref{prop.geometric-sen-semi-abeloid} and \cite[10.4]{he2024perfd}. Therefore, the assumption $f^{-1}(H_p\otimes \widehat{\overline{K}}(-1))=0$ implies that the geometric Sen morphism $\varphi_{\mathrm{Sen}}^{\mathrm{geo}}|_{T_pA/H,\overline{y}}$ is injective. Then, $F$ is pre-perfectoid by \cite[8.9]{he2024perfd} so that $\widetilde{\mathcal{A}}/\mathcal{H}$ is pre-perfectoid by Corollary \ref{cor:perfd-purity}.
\end{proof}

\bibliographystyle{amsalpha}
\bibliography{main,preprints}

\end{document}